\documentclass[onefignum,onetabnum]{siamart171218}
\usepackage{lipsum}
\usepackage{amsfonts}
\usepackage{graphicx}
\usepackage{epstopdf}
\usepackage{algorithmic}
\usepackage{braket}
\usepackage{tikz}
\usepackage{booktabs}   
\usetikzlibrary{quantikz2}
\usepackage{booktabs}
\usepackage{multirow}
\usepackage[caption=false]{subfig}
\ifpdf
  \DeclareGraphicsExtensions{.eps,.pdf,.png,.jpg}
\else
  \DeclareGraphicsExtensions{.eps}
\fi

\usepackage{bm}
\newcommand{\bu}{\bm{u}}

\newcommand{\bx}{\bm{x}}
\newcommand{\by}{\bm{y}}
\newcommand{\bY}{\bm{Y}}
\newcommand{\D}{\mathrm{d}}
\newcommand{\bmm}{\bm{m}}

\newsiamremark{remark}{Remark}
\newsiamremark{hypothesis}{Hypothesis}
\crefname{hypothesis}{Hypothesis}{Hypotheses}
\newsiamthm{claim}{Claim}

\headers{Quantum-assisted Bayesian PDE inversion}{D. An, Y. Li, P. Tang, and Y. Xiong}

\title{A quantum-assisted framework for PDE-based Bayesian inverse problems \thanks{Submitted to the editors. \funding{This work was funded by the National Key Research and Development Program of China (No. 2024YFE0102500), the Quantum Science and Technology - National Science and Technology Major Project (No. 2024ZD0301900), the National Natural Science Foundation of China (Nos. 62302346, 12622127, 12571413), the Hubei Provincial Natural Science Foundation of China (Nos. 2024AFA045), the Hubei major science and technology projects (No. ZDZX202600005), the Fundamental Research Funds for the Central Universities (No. 2042025kf0023), and the Fundamental Research Funds for the Central Universities, Peking University.}}}

\author{Dong An\thanks{Beijing International Center for Mathematical Research, Peking University, Beijing 100871, China (\email{dongan@pku.edu.cn}).}
\and Yinan Li\thanks{School of Artificial Intelligence \& Hubei Center for Applied Mathematics \& Hubei Key Laboratory of Computational Science, Wuhan University, Wuhan 430072, China; Wuhan Institute of Quantum Technology, Hubei 430206, China (\email{yinan.li@whu.edu.cn}).}
\and Pucheng Tang\thanks{School of Artificial Intelligence, Wuhan University, Wuhan 430072, China (\email{pucheng\_tang@whu.edu.cn}).}
\and Yunfeng Xiong\thanks{School of Mathematical Sciences \& Laboratory of Mathematics and Complex Systems, Ministry of Education, Beijing Normal University, Beijing 100875, China (\email{yfxiong@bnu.edu.cn}).}}
  
\usepackage{amsopn}

\makeatletter
\newcommand*{\addFileDependency}[1]{
  \typeout{(#1)}
  \@addtofilelist{#1}
  \IfFileExists{#1}{}{\typeout{No file #1.}}
}
\makeatother

\ifpdf
\hypersetup{
  pdftitle={Quantum for BayesOpt},
  pdfauthor={D. An, Y. Li, P. Tang, and Y. Xiong}
}
\fi

\begin{document}
\maketitle

\begin{abstract}
Quantum computing offers potential advantages for solving partial differential equations (PDEs). However, most existing quantum PDE solvers primarily focus  on preparing quantum states for solutions, while the efficient recovery of classical information from these states remains less explored. 
Motivated by the readout limitation, we propose a quantum-classical hybrid framework for Bayesian PDE inversion problems: The quantum processor evolves the PDE and evaluate the loss function with sampling noises, while the classical computer tunes the hyper-parameters in the Gaussian Process Regression to explore the next trial candidate. To match the quantum solvers for linear and semi-linear autonomous evolution PDEs, we suggest to use a normalized quantum-state loss as the data-misfit function and evaluate the new misfit by combining quantum PDE solvers with the Hadamard test, thereby allowing us to extract useful classical information using only a limited number of quantum state copies without reconstructing the full solution vector. The analysis of error propagation and overall complexity of loss evaluation under a prescribed accuracy shows that the new data-misfit function outperforms the conventional $L^2$-loss under quantum measurements. Quantum-circuit simulations of 1-D and 2-D linear convection–diffusion equations under approximate and finite-sampling loss evaluations, together with classical numerical experiments on a nonlinear forced viscous Burgers equation, demonstrate the feasibility of the proposed approach for parameter inversion even when the loss evaluations are affected by sampling noise. This framework may provide a viable quantum-assisted scheme for PDE-based inverse problems and elucidate the potential of quantum PDE algorithms in addressing a complete quantum-to-end optimization stack. 
\end{abstract}



\section{Introduction}
The recent progress of fault-tolerant quantum algorithms has witnessed a wide range of applications in solving partial differential equations (PDEs), arising from quantum physics \cite{berry2015simulating}, fluid dynamics \cite{liu2021efficient,JinLiu2024,jin2025inhomogeneous,AlipanahZhangYao2025,MengZhongXu2024}, engineering problems \cite{MontanaroPallister2016,AlkadriKharaziWhaleyMandadapu2025}, seismic modeling \cite{SchadeBöschHaplaFichtner2024,bosch2025quantum}
and  control theory \cite{JinLiu2026}. Leveraging laws of quantum mechanics such as superposition and entanglement, a quantum processor can potentially manipulate exponentially large degrees of freedom  \cite{feynman2018simulating,nielsen2010quantum}. This feature offers a promising approach for reducing both runtime and memory complexity for certain large-scale computational tasks \cite{arute2019quantum,divincenzo1995quantum}, especially for PDEs where the solution has large higher-order derivatives or large spatial dimension~\cite{SchadeBöschHaplaFichtner2024,AlkadriKharaziWhaleyMandadapu2025}. 

The basic design of quantum PDE algorithms is built on linear-algebraic and dynamical-simulation primitives, such as linear system solving \cite{harrow2009quantum,MoralesPiraSchleich2025}, Hamiltonian simulation \cite{berry2015simulating,Low2019Hamiltonian,chen2026timedependenthamiltoniansimulationoptimal}, and quantum signal processing (QSP) \cite{low2017optimal}. Related transformation frameworks, including quantum singular value transformation (QSVT) \cite{gilyen2019quantum} and quantum eigenvalue transformation (QET) \cite{low2026quantum}, further provide powerful tools for polynomial transformations of encoded operators. These foundational frameworks motivated the use of quantum algorithms as computational components for solving PDEs, including quantum linear system algorithms (QLSA) for discretized linear systems \cite{Berry2014HighOrder,berry2017quantum,childs2021high,DongLiXue2025}, the linear combination of Hamiltonian simulation (LCHS) techniques to simulate nonunitary dynamics \cite{an2023linear,low2025optimalquantumsimulationlinear}, and Schr\"odingerization \cite{jin2023quantum,JinLiu2024,jin2025schrodingerization,jin2025inhomogeneous,ma2026schrodingerization} which systematically transforms linear PDEs into Schr\"odinger-type dynamics suitable for quantum simulation \cite{jin2024quantum}. Several extensions have been developed for some specific nonlinear PDEs, based on either the Carleman linearization \cite{liu2021efficient,wu2025quantum,jennings2025quantumalgorithmsgeneralnonlinear}, the level set method \cite{Jin2024Quantumalgorithmsfornonlinear}, the Liouville formulation \cite{Li2026,JinLiu2026}, the Koopman/Koopman–von Neumann approaches \cite{Joseph2020KoopmanvonNeumann,jennings2026quantumkoopmanalgorithms}, and Young measures computation \cite{jin2026quantumalgorithmsyoungmeasures}. These endeavors establish profound connections between classical PDE discretizations and quantum primitives, while highlighting the potential of quantum algorithms for resolving the curse of dimensionality compared with their classical counterparts.

Despite these substantial algorithmic developments, most existing frameworks output the solutions of the underlying PDEs as quantum states rather than classical solution vectors as classical PDE solvers produced. For real-world applications, however, classical solution vectors are more easily utilized. A naive approach is to reconstruct the classical solution vector via quantum state tomography. Nevertheless, these postprocessing techniques come with exponential and prohibitive costs \cite{haah2016sample} which may destroy the quantum speed-up during the solution state preparation phase. Although it is possible to extract problem-specific quantities of interest, including physical observables  \cite{Jin2024Quantumalgorithmsfornonlinear,MengZhongXu2024,JinLiu2026} and  the misfit function in seismic applications \cite{SchadeBöschHaplaFichtner2024,bosch2025quantum}, from quantum states within an acceptable cost by several post-processing techniques, the apparent quantum speedups might still disappear when one takes the effect of solution accuracy into account \cite{MontanaroPallister2016}. The readout limitation urgently demands the development of a more flexible algorithmic framework that allows useful information to be extracted using only a limited number of quantum-state copies, and the task should have some tolerance on the noisy outputs. The latter is of great relevance when comparing the quantum and classical complexities of producing a classical answer to PDE-based inverse problems \cite{MontanaroPallister2016}.

In this paper, we aim to establish a noise-tolerant quantum-assisted framework for the Bayesian PDE inversion \cite{stuart2010inverse,tarantola2005inverse}, the target of which is to infer unknown physical parameters of PDEs  from observed data through repeated evaluations of the forward PDE model. It adopts a quantum-classical hybrid framework: the quantum processor evolves the PDE and outputs the loss function (with certain noises), while the classical computer tunes the hyper-parameters in the Gaussian Process Regression (GPR) to explore the next trial candidate. Our motivation comes from basic features of the Bayesian PDE inversion, as the central computational task is to evaluate a data-misfit function between observed data and forward solutions, instead of reconstructing the full solution vector  \cite{frazier2018bayesian,garnett2023bayesian}. Meanwhile, the GPR naturally utilizes the uncertainty inherent in noisy results for a richer exploration of the state space. Since objective-function evaluation is typically the dominant cost in classical computation and often requires repeated high-fidelity simulations of the forward PDE, quantum algorithms may provide an efficient alternative by estimating overlaps between quantum states without performing full amplitude readout \cite{nielsen2010quantum}. However, this perspective raises two key questions. First, since the quantum output encodes the solution as a quantum state, {\sl how should one define loss functions that are compatible with quantum-state measurements}? Second, {\sl how many copies of the quantum state, or equivalently how many circuit repetitions, are required to reconstruct a reliable loss estimate without eliminating the potential quantum advantage}?

To address these issues, we suggest replacing the conventional $L^2$-loss (which is called the physical loss hereafter) by  a normalized quantum-state loss as the misfit function within the Bayesian inversion framework. A quantum-state-based formulation is given to evaluate the new misfit function by combining quantum PDE solver \cite{an2023linear,liu2021efficient,an2025quantum,gutierrez2026quantum} with the Hadamard test to achieve an efficient approximation. This framework is compatible with different quantum solvers for linear evolution PDEs, and can be naturally extended to the semi-linear PDEs, where the Carleman linearization lifts the semi-discretized system to a truncated, higher-dimensional linear dynamical system \cite{liu2021efficient} and leads to a  representation analogous to that of the linear case. By analyzing error propagation through each computational step, we derive the complexity of loss evaluation under a prescribed accuracy and show that the normalized quantum-state loss outperforms the physical loss when the  quantum measurements are incorporated. Numerical experiments further demonstrate that Bayesian inversion can recover meaningful parameter estimates even when the loss evaluations are affected by sampling noise.  These results may elucidate the potential of quantum PDE solvers in addressing a complete quantum-to-end optimization stack, and may provide a viable quantum-assisted scheme for related PDE-based inverse problems like the seismic exploration \cite{SchadeBöschHaplaFichtner2024,bosch2025quantum}.

The remainder of the paper is organized as follows. In \Cref{sec:classical}, we review PDE-based inverse problems within a Bayesian inference framework. \Cref{sec:quantum} presents the quantum algorithms for PDE forward solution approximation and objective function evaluation. \Cref{sec:error_complex} provides the error and complexity analysis of the proposed framework. \Cref{sec:numerical} reports numerical results for both linear and nonlinear examples to validate the proposed method. Finally, \Cref{sec:conclusion} concludes the paper and discusses future research directions.

\section{Classical Bayesian PDE inversion} \label{sec:classical}
As a preliminary, we introduce the basic ingredients of the Bayesian optimization (BayesOpt) \cite{idier2013bayesian,chiappetta2026efficient}. The Bayesian component arises from GPR posterior distribution over the objective function, which quantifies surrogate uncertainty and guides the selection of candidate parameters through an acquisition function. 


\subsection{Notations}
Let us consider an evolution PDE with model parameters $\bmm =(m_1, \dots, m_d)$ with a feasible set $\mathcal{M}$, 
\begin{equation} 
    \label{eq:pde2}
    \frac{\partial u(\bx, t;\bmm)}{\partial t} = \mathcal{F}(u(\bx, t;\bmm), t;\bmm), \qquad u(\bx, 0;\bmm)=u_0(\bx),
\end{equation}
where the operator $\mathcal{F}$ may be linear or nonlinear. The forward problem \eqref{eq:pde2} establishes a mapping $\mathcal{G}: \mathcal{M} \to L^2(\mathbb{R})$: $\bmm \to \mathcal{G}(\bmm)$ and 
\begin{equation}
\mathcal{G}(\bmm) = u(\bx, T; \bmm) + \varepsilon,
\end{equation}
where $\varepsilon$ is the noise term.

Now we state an inversion problem:  Suppose we know the initial condition $\bu_0(\bx)$ and the terminal time $T > 0$, with little knowledge on the model parameters $\bmm$. We only have some discrete observations $\bm{u}_{\rm obs}=(u_{\rm obs}(\bx_1),\dots,u_{\rm obs}(\bx_{n_x}))^T$ at  time $T$. The target is to  infer the optimal parameter $\bmm^*$ to fit the observed data $\bm{u}_{\rm obs}$.

To this end, we need $u(\bx, t;\bmm)$ at discrete points $(\bx_1, \dots, \bx_{n_x})$ (either exact or numerical solutions). Denote the solution vector by
\begin{equation}
    \bu(t;\bmm)=(u(\bx_1,t;\bmm),\dots,u(\bx_{n_x},t;\bmm))^T\in\mathbb{R}^{n_x}.
\end{equation}
Now consider the $L^2$-misfit function (also called the physical loss),
\begin{equation}\label{physical_loss}
\begin{split}
\mathcal{L}_{\textup{phys}}(\bmm) &=  \Vert \mathcal{G}(\bmm)  - \bu_{\rm obs} \Vert^2  =  \sum_{i=1}^{n_x} | u(\bx_i, T; \bmm) + \varepsilon_i - u_{\rm obs}(\bx_i) |^2 \\
&= \|\mathcal{G}(\bmm)\|^2 + \|\bm{u}_{\rm obs}\|^2 - 2 \mathrm{Re}\braket{\mathcal{G}(\bmm),\bm{u}_{\rm obs}},
\end{split}
\end{equation}
and seek the optimal parameter $\bmm^\ast$
\begin{equation}
    \bm{\bmm}^\ast = \underset{\bmm\in\mathcal{M}}{\arg\min}\, \Vert \mathcal{G}(\bmm)  - \bu_{\rm obs} \Vert^2.
\end{equation}

Equivalently, this minimization problem can be reformulated as the maximization of an objective function $f(\bmm)$ defined through an exponential transformation:
\begin{equation}
    f(\bmm) = \exp\left( -\frac{1}{n_x\cdot\gamma} \|\mathcal{G}(\bmm)-\bm{u}_{\rm obs}\|^2 \right),
    \label{eq:misfit}
\end{equation}
where $\gamma>0$ is a fixed annealing parameter. 

\subsection{The framework of Bayesian inversion}

The general principle of the BayesOpt framework is trial and error. In each trial, each objective function evaluation often requires the operator $\mathcal{G}(\bmm)$ to be realized through a high-fidelity numerical simulation or a costly experimental procedure \cite{frazier2018bayesian,garnett2023bayesian}. This makes Bayesian optimization particularly suitable for settings in which only a limited number of objective function evaluations can be afforded.

\subsubsection{Surrogate model by GPR}
The starting point of BayesOpt is to construct a probabilistic surrogate model from previously evaluated samples and then use this surrogate to guide the selection of new trial points.

First, we collect a training dataset consisting of $n$ evaluated samples:
\[
    \mathcal{D}_n
    = \big\{(\bmm_1,f(\bmm_1)),(\bmm_2,f(\bmm_2)),\dots,(\bmm_n,f(\bmm_n))\big\}.
\]

A Gaussian process (GP) is a collection of random variables such that any finite collection of function values $\{f(\bmm_1),f(\bmm_2),\dots,f(\bmm_n)\}$ follows a multivariate Gaussian distribution. Accordingly, GP regression places the prior distribution \cite{rasmussen2003gaussian}
\[
    f(\bmm) \sim \mathcal{GP}\bigl(\mu(\bmm),k(\bmm,\bmm^{\prime})\bigr),
\]
where  $\mu(\bmm)$ is the mean function, which is typically chosen as $\mu(\bmm) \approx \frac{1}{n}\sum_{i=1}^n f(\bmm_i)$. $k(\bmm,\bmm^{\prime})$ is the covariance kernel that characterizes the correlation between function values at different input parameters. 

A common choice is the radial basis function (RBF) kernel augmented with a white-noise kernel:
\begin{equation}
\begin{split}
    k(\bmm,\bmm')
    &= \sigma_{\rm RBF}^2
    \exp\left(-\sum_{i=1}^d \frac{|m_i-m_i^{\prime}|^2}{2\ell_i^2}\right)
    + \sigma_{\rm white}^2\delta_{\bmm,\bmm'}.
\end{split}
\end{equation}
Here $\sigma_{\rm RBF}^2$ denotes the signal variance, $(\ell_1, \dots, \ell_d)$ are the length-scale parameters, $\sigma_{\rm white}^2$ is the noise variance, and $\delta_{\bmm,\bmm'}$ denotes the Kronecker delta, which is equal to one when the two inputs coincide and zero otherwise. 

Let $\bm{\eta}=\{\sigma_{\rm RBF},\sigma_{\rm white},\ell_1, \dots, \ell_d\}$ denote the collection of hyperparameters, which can be estimated by maximizing the log marginal likelihood or, equivalently, by minimizing the negative log marginal likelihood:
\[
    \bm{\eta}^{*} = \operatorname*{arg\,max}_{\bm{\eta}}
    \big[\log p\bigl(\bm{v}\mid\bmm_{1:n},\bm{\eta}\bigr)\big]
    \Leftrightarrow
    \operatorname*{arg\,min}_{\bm{\eta}}
    \big[-\log p\bigl(\bm{v}\mid\bmm_{1:n},\bm{\eta}\bigr)\big].
\]
where $\bm{v}=(f(\bmm_1),\ldots,f(\bmm_n) )^T$ and $\bmm_{1:n}=(\bmm_1,\ldots,\bmm_n)$. Based on the covariance matrix $K_{\bm{\eta}}$, whose entries are given by $(K_{\bm{\eta}})_{ij}=k(\bmm_i,\bmm_j)$ for $1\leq i,j\leq n$, the negative log marginal likelihood is
\[
    -\log p(\bm{v}\mid\bmm_{1:n},\bm{\eta})
    = \frac{1}{2}(\bm{v}-\bm{\mu})^T
    K_{\bm{\eta}}^{-1}(\bm{v}-\bm{\mu})
    + \frac{1}{2}\log\lvert K_{\bm{\eta}}\rvert
    + \frac{n}{2}\log(2\pi),
\]
where $\bm{\mu}=(\mu(\bmm_1),\dots,\mu(\bmm_n))^T$. The first term measures the data-fitting error with respect to the covariance structure specified by the GP, while the second term penalizes model complexity through the covariance matrix. 

\subsubsection{Bayesian optimization}
Once the GPR surrogate model has been constructed, an acquisition function $\alpha(\bmm):\mathbb{R}^d\rightarrow\mathbb{R}$ is used to select the next candidate parameter $\bmm_{\rm post}$ based on the GP posterior distribution. A commonly used acquisition function is expected improvement (EI), which quantifies the expected improvement over the best objective value observed so far. Here, we first consider the ideal setting in which the noise-free observation $\bu_{\rm obs}$ is available.

At each iteration, the current best point is the one with the largest observed objective value. We define  
$f^*_n = \max(f(\bmm_1),f(\bmm_2),\dots,f(\bmm_n))$
as the current best value. Suppose that one additional evaluation is permitted at any location $\bmm$, where $f(\bmm) \sim\mathcal{N}(\mu_n(\bmm),\sigma^2_n(\bmm))$. The EI function is defined by
\begin{equation} \label{eq:ei}
    \alpha_{\rm EI}(\bmm) := \mathbb{E}_n \big[ [f(\bmm) - f_n^*]^+ \big] \quad \mathrm{where} \quad a^+=\max(a,0).
\end{equation}

This criterion favors points that are expected to improve upon the best observed objective value. Let $y$ denote a realization of the posterior predictive random variable $f(\bmm) |\mathcal{D}_n$. Under the GP posterior distribution, \eqref{eq:ei} can be written as
\[
    \alpha_{\mathrm{EI}}(\bmm) = \int_{f_n^*}^{\infty}(y - f_n^*) \cdot \frac{1}{\sigma_n(\bmm)\sqrt{2\pi}} \exp\left( -\frac{(y - \mu_n(\bmm))^2}{2\sigma_n^2(\bmm)} \right) \D y.
\]
Through integration by parts, EI admits the closed-form expression:
\[
    \alpha_{\mathrm{EI}}(\bmm) = \Delta_n(\bmm)\Phi\!\left(\frac{\Delta_n(\bmm)}{\sigma_n(\bmm)}\right) + \sigma_n(\bmm)\phi\!\left(\frac{\Delta_n(\bmm)}{\sigma_n(\bmm)}\right),
\]
where $\Delta_n(\bmm) = \mu_n(\bmm) - f_n^*$, and $\phi(\cdot)$ and $\Phi(\cdot)$ denote the standard Gaussian probability density function and cumulative distribution function, respectively.

Finally, the next evaluation point is selected by maximizing the EI function: 
\[
    \bmm_{\rm post} = \underset{\bmm\in\mathcal{M}}{\arg\max}\, \alpha_{\rm EI}(\bmm).
\]

After each evaluation, the dataset $\mathcal{D}_n$ and the GPR surrogate are updated. This iterative procedure continues until a prescribed stopping criterion is satisfied, such as reaching the maximum number of iterations or convergence of the acquisition function.

Based on the BayesOpt surrogate model and parameter inference procedure, the overall Bayesian inversion framework can be summarized as follows:
\begin{description}
    \item[Step 1.] Construct a GP model $f(\bmm) \sim \mathcal{GP}\big(\mu(\bmm), k(\bmm,\bmm^{\prime})\big)$ from the dataset $\mathcal{D}_n$ to probabilistically characterize the objective function.
    
    \item[Step 2.] Train the GPR hyperparameters by minimizing the negative log marginal likelihood.
    
    \item[Step 3.] Construct an acquisition function to select the next evaluation point $\bmm_{\rm post}$ by quantifying the expected utility of each candidate point.

    \item[Step 4.] Once $\bmm_{\rm post}$ has been selected, evaluate $f(\bmm_{\rm post})$ using a high-fidelity solver, augment the dataset according to $\mathcal{D}_{n+1} =  \mathcal{D}_n \cup \{(\bmm_{\rm post}, f(\bmm_{\rm post}))\}$, and repeat the procedure from Step 1 until the maximum number of iterations is reached.
\end{description}

\section{Quantum-assisted Bayesian PDE inversion} \label{sec:quantum}

\begin{figure}[ht]
    \centering
    \includegraphics[width=\linewidth,height=0.6\linewidth]{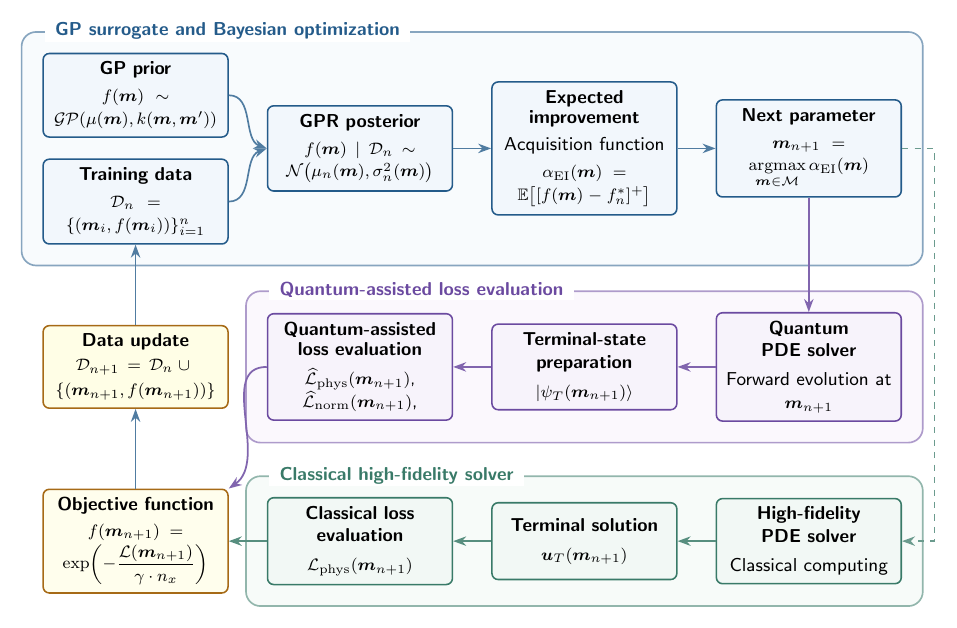}
    \caption{Workflow of the proposed quantum-assisted Bayesian PDE inversion framework.}
    \label{fig:quantum_bayesian_workflow}
\end{figure}

For classical BayesOpt, the main computational burden is the repeated evaluation of the objective function in \eqref{eq:misfit}, with each evaluation requiring a high-fidelity forward PDE solver \cite{tarantola2005inverse}. Regarding the potential advantage of quantum algorithms for PDEs, we explore a quantum-assisted strategy for evaluating the objective function through quantum-state operations without explicitly reconstructing the complete solution, while the surrogate model by GPR is still constructed on a classical computer. The overall workflow is illustrated in \cref{fig:quantum_bayesian_workflow}.

Most existing quantum PDE solvers prepare quantum states containing approximations to the terminal solution $\bu_T(\bmm)$, possibly together with solver-dependent ancillary and garbage components \cite{harrow2009quantum,MontanaroPallister2016,childs2017quantum,liu2021efficient,JinLiu2024,Li2026}. Here, we briefly summarize three representative settings, with further details provided in the supplementary note.

\begin{enumerate}

\item[(1)] For the linear autonomous equation $\frac{\D\bu(t;\bmm)}{\D t}=A(\bmm)\bu(t;\bmm)$, the terminal solution is $\bu_T(\bmm)=e^{TA(\bmm)}\bu_0$. The corresponding matrix-exponential action can, for example, be approximated by the {P}adé-based algorithm \cite{DongLiXue2025}.

\item[(2)] For the linear non-autonomous equation, $\frac{\D\bu(t;\bmm)}{\D t}=A(t;\bmm)\bu(t;\bmm)+\boldsymbol{b}(t;\bmm)$, let
$\Phi_{\bmm}(t,s)=\mathcal{T}\exp\bigl(\int_s^t A(\tau;\bmm)\,\D\tau\bigr)$
denote the evolution operator generated by $A(t;\bmm)$, where $\mathcal{T}$ denotes time ordering. The variation-of-constants formula gives
\begin{equation*}
    \bu_T(\bmm) = \Phi_{\bmm}(T,0)\bu_0 + \int_0^T \Phi_{\bmm}(T,s)\boldsymbol{b}(s;\bmm)\,\D s.
\end{equation*}
Such inhomogeneous non-autonomous systems can be tackled by quantum computing using linear combinations of Hamiltonian simulation (LCHS) \cite{an2023linear} or the Schr{\"o}dingerization approach under certain assumptions \cite{jin2024quantum}.

\item[(3)] After spatial discretization, suppose that the nonlinear PDE gives rise to the possibly polynomial system
$\frac{\D\bu(t;\bmm)}{\D t} =F_0(t;\bmm) + \sum_{\alpha=1}^{\infty} F_{\alpha}(t;\bmm)\bu(t;\bmm)^{\otimes\alpha}$, where $F_0(t;\bmm)\in\mathbb{R}^{{n_x}}$ represents the external forcing and $F_{\alpha}(t;\bmm)\in\mathbb{R}^{{n_x}\times {n_x^{\alpha}}}$. Carleman linearization \cite{liu2021efficient} introduces the lifted vector $\bY=(\bu,\bu^{\otimes2},\dots)^T$, for which the nonlinear system is represented by the infinite-dimensional inhomogeneous linear system
\begin{equation*}
 \frac{\D}{\D t}
\begin{pmatrix}
    \bu \\ \bu^{\otimes 2} \\ \bu^{\otimes 3} \\ \vdots
\end{pmatrix} =
\begin{pmatrix}
A_1^1(t;\bmm) & A_1^2(t;\bmm)  & \cdots \\
A_2^1(t;\bmm) & A_2^2(t;\bmm)  & \cdots \\
0 & A_3^2(t;\bmm)  & \cdots \\
\vdots & \vdots  & \ddots
\end{pmatrix}
\begin{pmatrix}
    \bu \\ \bu^{\otimes 2} \\ \bu^{\otimes 3} \\ \vdots
\end{pmatrix} +
\begin{pmatrix}
    F_0(t;\bmm) \\  0 \\ 0 \\ \vdots
\end{pmatrix}.
\end{equation*}
Here, for \(j\geq 1\), the block mapping \(\bu^{\otimes k}\) to the evolution equation for \(\bu^{\otimes j}\) is
\begin{equation*}
    A_j^k(t;\bmm) =  \sum_{\ell=0}^{j-1} I_{n_x}^{\otimes \ell}
    \otimes F_{k-j+1}(t;\bmm) \otimes I_{n_x}^{\otimes (j-1-\ell)},
    \quad k\geq \max\{1,j-1\},
\end{equation*}
and \(A_j^k(t;\bmm)=0\) for \(k<\max\{1,j-1\}\). In particular, the blocks \(A_j^{j-1}(t;\bmm)\) for \(j\geq 2\) are induced by the forcing term \(F_0(t;\bmm)\), while the forcing term in the original equation for \(\bu\) appears in the inhomogeneous vector. Truncating the lifted vector at order \(N\) gives
\(\bY_N=(\bu,\bu^{\otimes 2},\dots,\bu^{\otimes N})^T\), whose first-order component provides an approximation to \(\bu_T(\bmm)\).

\end{enumerate}

Quantum PDE solvers considered here require coherent access to the initial data. We assume that the normalized amplitude encoding $\ket{\bu_0}$ can be prepared by an initial-state-preparation oracle $O_{\rm prep}:\ket{0}\rightarrow\ket{\bu_0}$. Although the quantum PDE solvers applicable to these three settings use different internal constructions, the subsequent loss evaluation requires only a common coherent-output form \cite{an2023linear,jin2023quantum,childs2017quantum}. 

Let $\bu_T(\bmm)=\bigl(u(\bx_1,T;\bmm),\dots,u(\bx_{n_x},T;\bmm)\bigr)^T$ denote the exact terminal solution and $\widetilde{\bu}_T(\bmm)$ denote the unnormalized approximation encoded by the selected quantum solver. Suppose that the forward solver produces an approximation satisfying $\left\|\widetilde{\bu}_T(\bmm)-\bu_T(\bmm)\right\| \leq \epsilon_{\mathrm{app}}$. The corresponding normalized amplitude encoding of the approximate terminal solution $\ket{\widetilde{\bu}_T(\bmm)}$ is defined by
\begin{equation*}
    \ket{\widetilde{\bu}_T(\bmm)} = \frac{1}{\left\|\widetilde{\bu}_T(\bmm)\right\|} \sum_{j=0}^{{n_x}-1} \widetilde{u}_{T,j+1}(\bmm)\ket{j},
\end{equation*}
where $\widetilde{u}_{T,j+1}(\bmm)$ is the $(j+1)$th component of $\widetilde{\bu}_T(\bmm)=\bigl(\widetilde{u}_{T,1}(\bmm),\dots,\widetilde{u}_{T,{n_x}}(\bmm)\bigr)^T$.

Suppose that $U_{\mathrm{solver}}(\bmm)$ denotes the complete unitary implementation of the selected quantum PDE solver, including the initial-state preparation and all solver-specific coherent operations, and an $a$-qubit auxiliary register identifies the desired PDE solution branch. Its action is assumed to satisfy
\begin{equation} \label{eq:output_state}
    \ket{\psi_T(\bmm)} := U_{\mathrm{solver}}(\bmm)\ket{0}
    = \frac{\left\|\widetilde{\bu}_T(\bmm)\right\|} {\lambda(\bmm)\left\|\bu_0\right\|} \ket{0}^{\otimes a}\ket{\widetilde{\bu}_T(\bmm)} + \ket{\Phi_T^\perp(\bmm)},
\end{equation}
where $\ket{\Phi_T^\perp(\bmm)}$ denotes the generally garbage component satisfying $P_u\ket{\Phi_T^\perp(\bmm)}=0$, with $P_u:=\ket{0}\!\bra{0}^{\otimes a}\otimes I$ denoting the projector onto the successful branch of the quantum solver. The solver-dependent scaling factor $\lambda(\bmm)>0$ incorporates the solver-specific scaling required to encode the terminal solution, ensuring that $\|\widetilde{\bu}_T(\bmm)\| / (\lambda(\bmm)\|\bu_0\|) \leq 1.$


\subsection{Loss function evaluation via quantum measurement}
A key challenge comes from the fact that quantum output $\ket{\psi_T}$ contains the desired auxiliary qubits with the orthogonal garbage component $\ket{\Phi^\perp_T(\bmm)}$ in \eqref{eq:output_state}. Consequently, neither the full classical solution vector nor its norm can be accessed directly from the quantum state. In this subsection, we show how the Hadamard test and the success probabilities of the quantum PDE solver circuit can be combined to evaluate the loss function.
Throughout this section, we distinguish among three notations for each loss function.
\begin{enumerate}
    \item[(1)] $\mathcal{L}(\bmm)$: the exact loss constructed from the exact solution $\bu_T(\bmm)$;
    \item[(2)] $\widetilde{\mathcal{L}}(\bmm)$: the infinite-shot quantum loss constructed from the deterministic solver approximation $\ket{\widetilde{\bu}_T(\bmm)}$ and the exact quantum probabilities;
    \item[(3)] $\widehat{\mathcal{L}}(\bmm)$: the finite-shot estimator obtained by replacing the exact quantum probabilities with their empirical frequencies.
\end{enumerate}

\subsubsection{Physical loss function}

The most straightforward choice is the physical loss $\mathcal{L}_{\textup{phys}}(\bmm) =  \Vert \bu_T(\bmm)  - \bu_{\rm obs} \Vert^2$ adopted in the classical PDE inversion,
which retains both the overall magnitude and the normalized profile of the terminal solution. 

To evaluate $\widetilde{\mathcal{L}}_{\rm phys}(\bmm)$ using quantum algorithms, we define the normalized amplitude encoding of the observation $\bu_{\rm obs}=(u_{{\rm obs},1},\dots,u_{{\rm obs},{n_x}})^T$ as $\ket{\bu_{\rm obs}}.$ The corresponding observation state in the successful branch is 
\begin{equation} \label{eq:obs_state_u} 
    \ket{\psi_{\rm obs}} = \ket{0}^{\otimes a}
    \otimes \left( \sum_{j=0}^{{n_x}-1} \frac{1}{\|\bu_{\rm obs}\|} u_{{\rm obs},j+1}\ket{j} \right). 
\end{equation}

We assume access to an observation-state oracle $O_{\rm obs}$ satisfying $\ket{\psi_{\rm obs}}=O_{\rm obs}\ket{0}$. Such an assumption is particularly relevant when the observational data possess exploitable structures, for example, when they are obtained from a sparse set of sensor measurements \cite{zhang2022quantum,ramacciotti2024simple,mao2024toward,li2025nearly,alexanderian2021optimal}, as commonly considered in Bayesian inverse problems \cite{alexanderian2021optimal,alexanderian2014optimal}. Recall that the complete solver output is $\ket{\psi_T(\bmm)}=U_{\rm solver}(\bmm)\ket{0}$. Only the component of $\ket{\psi_T(\bmm)}$ with the auxiliary registers in  $\ket{0}^{\otimes a}$ contributes to the overlap with $\ket{\psi_{\rm obs}}$. The remaining garbage component $\ket{\Phi_T^\perp(\bmm)}$ is orthogonal to the observation state and therefore does not contribute to the overlap in the data misfit function.

To estimate the overlap, we define the unitary operator $W_u(\bmm)=O_{\rm obs}^\dagger U_{\rm solver}(\bmm)$. Its expectation value with respect to the all-zero initial state satisfies $\bra{0}W_u(\bmm)\ket{0} = \bra{0}O_{\rm obs}^\dagger U_{\rm solver}(\bmm)\ket{0} = \braket{\psi_{\rm obs}|\psi_T(\bmm)}$. The real part of $\bra{0}W_u(\bmm)\ket{0}$ can be estimated using the Hadamard test \cite{nielsen2010quantum}. Given an auxiliary qubit initialized in $\ket{0}_H$, we apply the controlled operation $c\text{-}W_u(\bmm)$, where $W_u(\bmm)$ is applied to the remaining registers conditioned on the auxiliary qubit being in $\ket{1}_H$. The circuit follows:
\begin{equation*} 
\begin{aligned}
    \ket{0}_H\ket{0}
    &\xrightarrow{H\otimes I}
    \frac{\ket{0}_H+\ket{1}_H}{\sqrt 2}\ket{0}
    \xrightarrow{c-W_u(\bmm)}
    \frac{1}{\sqrt 2}
    \big( \ket{0}_H\ket{0} +\ket{1}_H W_u(\bmm)\ket{0} \big) \\
    &\xrightarrow{H \otimes I}
    \frac{1}{2}\ket{0}_H
    \big( \ket{0} + W_u(\bmm)\ket{0} \big) +
    \frac{1}{2}\ket{1}_H
    \big( \ket{0} - W_u(\bmm)\ket{0} \big).
\end{aligned}
\end{equation*}
Consequently, the probability of measuring the auxiliary qubit to be in state $\ket{0}_H$ is
\begin{equation} \label{eq:u_hada}
\begin{aligned}
    p^u_H(\bmm)
    &= \left\| \frac{1}{2}\left( \ket{0} + W_u(\bmm)\ket{0} \right)
    \right\|^2 \\
    &=  \frac{1}{4}
    \left( \braket{0|0} + \bra{0}W_u(\bmm)\ket{0} + \bra{0}W_u^\dagger(\bmm)\ket{0} + \bra{0}W_u^\dagger(\bmm)W_u(\bmm)\ket{0} \right) \\
    &= \frac{1}{2} \left( 1 + \mathrm{Re} \bra{0}W_u(\bmm)\ket{0}
    \right) = \frac{ 1+ \mathrm{Re} \braket{\psi_{\rm obs}|\psi_T(\bmm)} }{2}.
\end{aligned}
\end{equation}
Therefore, the inner product can be estimated as
\[
    \mathrm{Re} \braket{\psi_{\rm obs}|\psi_T(\bmm)} 
    = \frac{ \mathrm{Re} \braket{ \bu_{\rm obs},
    \widetilde{\bu}_T(\bmm)}}{ \lambda(\bmm) \|\bu_0\| \|\bu_{\rm obs}\|}
    = 2p^u_H(\bmm)-1  .
\]
For a constant confidence level, this overlap can be estimated to a prescribed additive accuracy $\epsilon$ using $\mathcal{O}(1/\epsilon^2)$ repetitions of the Hadamard test.

We additionally measure the projected success probability,
\begin{equation}
    p_{\rm succ}(\bmm) := \bra{\psi_T(\bmm)} P_u \ket{\psi_T(\bmm)}
    = \left\| P_uU_{\mathrm{solver}}(\bmm)\ket{0} \right\|^2
    = \frac{ \left\|\widetilde{\bu}_T(\bmm)\right\|^2}{ \lambda(\bmm)^2\left\|\bu_0\right\|^2 }.
\label{eq:project_succ}
\end{equation}
Experimentally, $p_{\rm succ}(\bmm)$ is obtained by executing the forward-solver circuit and recording the frequency with which the auxiliary registers are measured in $\ket{0}^{\otimes a}$. The squared norm can therefore be reconstructed as $ \|\widetilde{\bu}_T(\bmm)\|^2 = \lambda(\bmm)^2\|\bu_0\|^2p_{\rm succ}(\bmm).$ And the real part of the inner product is obtained from the Hadamard test statistics:
\[
    \mathrm{Re} \braket{\bu_{\rm obs}, \widetilde{\bu}_T(\bmm) }
    =\lambda(\bmm) \|\bu_0\| \|\bu_{\rm obs} \|\big(2p^u_H(\bmm)-1\big).
\]

After eliminating the normalization factor by combining these two quantities, the infinite-shot physical loss function is
\begin{equation} \label{eq:loss_phy}
    \widetilde{\mathcal L}_{\rm phys}(\bmm) :=
    \lambda(\bmm)^2\|\bu_0\|^2p_{\rm succ}(\bmm) + \|\bu_{\rm obs}\|^2 - 2\lambda(\bmm) \|\bu_0\| \|\bu_{\rm obs}\| \big( 2p^u_H(\bmm)-1 \big).
\end{equation}

Thus, the conventional physical loss can be reconstructed from the projected quantum measurements using the known quantities $\lambda(\bmm)$, $\|\bu_0\|$ and $\|\bu_{\rm obs}\|$. To achieve
$|\widetilde{\mathcal L}_{\rm phys}(\bmm)-\widehat{\mathcal{L}}_{\rm phys}(\bmm)|\leq\epsilon$, the leading total sampling complexity scales as
\[
    N_{\rm total} = \mathcal{O}\!\left( \frac{ \lambda(\bmm)^4\|\bu_0\|^4
    + \lambda(\bmm)^2 \|\bu_0\|^2 \|\bu_{\rm obs}\|^2} {\epsilon^2} \right).
\]

In other words, the quantum estimation requires recovering the physical scaling of the terminal solution, and the associated norm factors amplify finite-shot errors. This motivates the introduction of normalized losses that are more naturally compatible with quantum-state outputs. By characterizing the discrepancy through normalized quantum-state overlaps, these losses can be estimated without explicitly recovering the physical scaling of the terminal solution and may therefore require fewer quantum measurements.

\subsubsection{Normalized quantum-state loss in the physical system}
In order to avoid the explicit recovery of $\lambda(\bmm)\|\bu_0\|$ or $\|\widetilde{\bu}_T(\bmm)\|$, we introduce the normalized quantum-state loss in the physical subspace, which utilize the measurement quantities $p^u_H(\bmm)$ and $p_{\rm succ}(\bmm)$ without reconstructing the physical norm of $\widetilde{\bu}_T(\bmm)$. Assuming that $\bu_{\rm obs}\neq \bm{0}$ and $\widetilde{\bu}_T(\bmm)\neq \bm{0}$, the exact inner product of two normalized vectors can be estimated from the ratio of the two experimentally accessible quantities as:
\[
    \frac{\text{Re}\langle \bu_{\rm obs}, \widetilde{\bu}_T(\bmm) \rangle}{\|\bu_{\rm obs}\|\|\widetilde{\bu}_T(\bmm)\|}  = {\frac{\text{Re}\langle \bu_{\rm obs}, \widetilde{\bu}_T(\bmm) \rangle}{\lambda(\bmm)\|\bu_{\rm obs}\|\|\bu_0\|}} \big/ {\frac{\|\widetilde{\bu}_T(\bmm)\|}{\lambda(\bmm)\|\bu_0\|}}=  \frac{2p^u_H(\bmm) - 1}{\sqrt{p_{\rm succ}(\bmm)}}.
\]

For the quantum estimator, we further require $p_{\rm succ}(\bmm)>0$. This yields the normalized loss defined on the physical variable:
\begin{equation} \label{eq:loss_norm_u}
     \widetilde{\mathcal L}^u_{\rm norm}(\bmm) := 2-2 \frac{\mathrm{Re}\langle \bu_{\rm obs}, \widetilde{\bu}_T(\bmm) \rangle}{\|\bu_{\rm obs}\|\,\|\widetilde{\bu}_T(\bmm)\|} = 
    2 - \frac{4p^u_H(\bmm) - 2}{\sqrt{p_{\rm succ}(\bmm)}}.
\end{equation}

Its finite-shot sampling complexity for a prescribed additive accuracy $|\widetilde{\mathcal L}^u_{\rm norm}-\widehat{\mathcal L}^u_{\rm norm}|\leq\epsilon$ scales as
$
    N_{\rm total} = \mathcal{O}\left(\frac{1}{p_{\rm succ}(\bmm) \epsilon^2}\right).
$

The normalized loss compares only the direction, or normalized profile, of the terminal solution and the observation. It therefore discards their overall norm information: two vectors that differ only by a positive multiplicative factor yield zero normalized loss. This loss is consequently appropriate when the inversion is primarily determined by the normalized spatial profile, or when the reduced measurement dependence is more important than recovering the absolute solution scale.

\subsubsection{Normalized quantum-state loss in the lifted Carleman space}

For a nonlinear problem by Carleman linearization, the normalized loss can also be defined and evaluated directly in the truncated lifted space. Let $\bY_{N,T}(\bmm)=(\bu_T(\bmm),\bu_T(\bmm)^{\otimes 2},\ldots,\bu_T(\bmm)^{\otimes N})^T\in\mathbb{R}^{D_N}$, where $D_N=\sum_{j=1}^{N}n_x^j$ denotes the $N$ order lift of the exact terminal solution, and let $\widetilde{\bY}_{N,T}(\bmm)$ denote the corresponding unnormalized approximation produced by the quantum PDE solver. Carleman-based quantum algorithms encode the truncated lifted system as a linear evolution and prepare a coherent state containing the lifted terminal solution in a designated successful subspace \cite{liu2021efficient}.

Following the general output representation in \eqref{eq:output_state}, we assume that the quantum solver for the truncated Carleman system produces
\begin{equation*}
    \ket{\psi_T^Y(\bmm)}:=U_{\rm solver}^Y(\bmm)\ket{0}=\frac{\|\widetilde{\bY}_{N,T}(\bmm)\|}{\lambda_Y(\bmm)\|\bY_{N,0}\|}\ket{0}^{\otimes a_Y}\ket{\widetilde{\bY}_{N,T}(\bmm)}+\ket{\Phi_T^{Y,\perp}(\bmm)},
\end{equation*}
where $\bY_{N,0}=(\bu_0,\bu_0^{\otimes 2},\ldots,\bu_0^{\otimes N})^T\in\mathbb{R}^{D_N}$, $a_Y$ is the number of auxiliary qubits identifying the successful lifted-solution branch, and $\lambda_Y(\bmm)$ denotes the scaling factor in Carleman system. The normalized amplitude encoding of the approximate lifted solution $\widetilde{\bY}_{N,T}(\bmm)=\big(\widetilde{Y}_{N,T,1}(\bmm), \dots,\widetilde{Y}_{N,T,D_N}(\bmm)\big)^T$ follows
\[
    \ket{\widetilde{\bY}_{N,T}(\bmm)}=\frac{1}{\|\widetilde{\bY}_{N,T}(\bmm)\|}\sum_{\ell=0}^{D_N-1}\widetilde{Y}_{N,T,\ell+1}(\bmm)\ket{\ell}.
\]

The loss-evaluation procedure developed in the preceding subsections remains applicable by redefining the observation state in the lifted space. Define the lifted observation vector as $\bY_{\rm obs}=(\bu_{\rm obs},\bu_{\rm obs}^{\otimes 2},\ldots,\bu_{\rm obs}^{\otimes N})^T\in\mathbb{R}^{D_N}$, and its corresponding normalized amplitude-encoded state in the full space is given by
\begin{equation} \label{eq:obs_state}
    \ket{\psi_{\rm obs}^Y} :=
    \ket{0}^{\otimes a_Y}\otimes \left( \frac{1}{\|\bY_{\rm obs}\|} \sum_{\ell=0}^{D_N-1}Y_{{\rm obs},\ell+1}\ket{\ell} \right) =O_{\rm obs}^Y\ket{0},
\end{equation}
where $O_{\rm obs}^Y$ denotes the associated state-preparation oracle.

Define $W_Y(\bmm)=(O_{\rm obs}^Y)^\dagger U_{\rm solver}^Y(\bmm)$. Applying the same Hadamard test introduced above gives
\begin{equation} \label{eq:y_hada}
    \mathrm{Re}\braket{\psi_{\rm obs}^Y|\psi_T^Y(\bmm)}=\frac{\mathrm{Re}\langle\bY_{\rm obs},\widetilde{\bY}_{N,T}(\bmm)\rangle}{\lambda_Y(\bmm)\|\bY_{N,0}\|\|\bY_{\rm obs}\|}
    = 2p_H^Y(\bmm)-1.
\end{equation}

Let $P_Y=(\ket{0}\bra{0})^{\otimes a_Y}\otimes I$ denote the projector onto the successful lifted-solution subspace. The corresponding success probability is
\begin{equation} \label{eq:carleman_succ}
    p^Y_{\rm succ}(\bmm):=\bra{\psi_T^Y(\bmm)}P_Y\ket{\psi_T^Y(\bmm)}=\frac{\|\widetilde{\bY}_{N,T}(\bmm)\|^2}{\lambda_Y(\bmm)^2\|\bY_{N,0}\|^2}.
\end{equation}

Consequently, taking the ratio of these two experimentally accessible quantities eliminates both the solver-dependent normalization factor and the norm of the initial lifted vector:
\[
    \frac{\mathrm{Re}\langle\bY_{\rm obs},\widetilde{\bY}_{N,T}(\bmm)\rangle}{\|\bY_{\rm obs}\|\|\widetilde{\bY}_{N,T}(\bmm)\|}
    = \frac{2p_H^Y(\bmm)-1}{\sqrt{p^Y_{\rm succ}(\bmm)}}
    =\frac{\mathrm{Re}\langle\bY_{\rm obs},\widetilde{\bY}_{N,T}(\bmm)\rangle}{\lambda_Y(\bmm)\|\bY_{N,0}\|\|\bY_{\rm obs}\|} \big/ \frac{\|\widetilde{\bY}_{N,T}(\bmm)\|}{\lambda_Y(\bmm)\|\bY_{N,0}\|} .
\]

The infinite-shot normalized lifted loss $ \widetilde{\mathcal{L}}_{\rm norm}^{Y}(\bmm)$ in the truncated Carleman space is therefore given by
\begin{equation} \label{eq:loss_norm_Y}
    \widetilde{\mathcal{L}}_{\rm norm}^{Y}(\bmm):=2-2\frac{\mathrm{Re}\langle\bY_{\rm obs},\widetilde{\bY}_{N,T}(\bmm)\rangle}{\|\bY_{\rm obs}\|\|\widetilde{\bY}_{N,T}(\bmm)\|}=2-\frac{4p_H^Y(\bmm)-2}{\sqrt{p^Y_{\rm succ}(\bmm)}}.
\end{equation}
The finite-shot estimator $\widehat{\mathcal{L}}_{\rm norm}^{Y}(\bmm)$ is obtained by replacing $p_H^Y(\bmm)$ and $p^Y_{\rm succ}(\bmm)$ with their empirical frequencies. For a fixed confidence level, achieving an additive accuracy $|\widehat{\mathcal{L}}_{\rm norm}^{Y}(\bmm)-\widetilde{\mathcal{L}}_{\rm norm}^{Y}(\bmm)|\leq\epsilon$ requires a total number of circuit repetitions scaling as $N_{\rm total}=\mathcal{O}(1/(p^Y_{\rm succ}(\bmm)\epsilon^2))$.

The same coherent Carleman output also supports loss evaluation in the original physical space. Let $\Pi_1:\mathbb{R}^{D_N}\rightarrow\mathbb{R}^{n_x}$ denote the extraction operator for the first-order Carleman component, so that $\bu_T(\bmm)=\Pi_1\bY_{N,T}(\bmm)$ and $\widetilde{\bu}_T(\bmm)=\Pi_1\widetilde{\bY}_{N,T}(\bmm)$. In the quantum implementation, we introduce a flag register $F$ and a reversible oracle $O_u$ acting on the lifted-system register and the flag register, while leaving the solver ancillas unchanged, such that
\begin{equation*}
    O_u\ket{\ell}\ket{0}_F=
\begin{cases}
    \ket{\ell}\ket{1}_F, & \ell=0,\ldots,{n_x}-1,\\
    \ket{\ell}\ket{0}_F, & \mathrm{otherwise}.
\end{cases}
\end{equation*}
Thus, $O_u$ coherently flags the first-order Carleman component without modifying its amplitudes. The projected success probability is obtained by measuring the flag register in $\ket{1}_F$ together with all solver ancillas in the all-zero state, and satisfies
\begin{equation*}
\begin{aligned}
    p_{\rm succ}(\bmm) := \left\| \big[ \big(\bra{0}^{\otimes a_Y}\otimes I\big)\otimes\bra{1}_F \big] O_u 
    \left( \ket{\psi_T^Y(\bmm)}\otimes\ket{0}_F \right) \right\|^2  = \frac{\|\widetilde{\bu}_T(\bmm)\|^2} {\lambda_Y(\bmm)^2\|\bY_{N,0}\|^2}.
\end{aligned}
\end{equation*}

To estimate the corresponding physical-space overlap, define the embedded observation vector $\bY_{\rm obs}^u=(\bu_{\rm obs},0,\ldots,0)^T\in\mathbb{R}^{D_N}$ and prepare its normalized observation state in the successful physical-solution subspace. Applying the Hadamard test construction from the preceding subsections gives
\[
    2p^u_H(\bmm)-1=\frac{\mathrm{Re}\langle\bu_{\rm obs},\widetilde{\bu}_T(\bmm)\rangle}{\lambda_Y(\bmm)\|\bY_{N,0}\|\|\bu_{\rm obs}\|}.
\]
Consequently, the normalized loss in the physical system can be evaluated from the lifted Carleman output as
\[
    \widetilde{\mathcal{L}}_{\rm norm}^{u}(\bmm)= 2- 2\frac{\mathrm{Re}\langle\bu_{\rm obs},\widetilde{\bu}_T(\bmm)\rangle}{\|\bu_{\rm obs}\|\|\widetilde{\bu}_T(\bmm)\|}
    =2-\frac{4p^u_H(\bmm)-2}{\sqrt{p_{\rm succ}(\bmm)}},
\]
whereas the corresponding approximation of the physical loss reads
\[
    \widetilde{\mathcal{L}}_{\rm phys}(\bmm)=\lambda_Y(\bmm)^2\|\bY_{N,0}\|^2p_{\rm succ}(\bmm)+\|\bu_{\rm obs}\|^2-2\lambda_Y(\bmm)\|\bY_{N,0}\|\|\bu_{\rm obs}\|\bigl(2p^u_H(\bmm)-1\bigr).
\]

Finally, the success probabilities associated with the full lifted state and the projected physical component satisfy
\[
    p^Y_{\rm succ}(\bmm)=p_{\rm succ}(\bmm)+\frac{\|(I-\Pi_1^\dagger\Pi_1)\widetilde{\bY}_{N,T}(\bmm)\|^2}{\lambda_Y(\bmm)^2\|\bY_{N,0}\|^2}\geq p_{\rm succ}(\bmm).
\]
Therefore, the normalized lifted loss may require fewer circuit repetitions than the projected normalized loss. However, it compares the complete truncated Carleman hierarchy and hence incorporates all tensor-product components $\bu^{\otimes j}$ up to order $N$, rather than measuring the discrepancy only in the original physical variable $\bu$. A detailed comparison of the normalized lifted loss, the projected normalized loss, and the physical loss reconstructed from quantum measurements is provided in \cref{tab:loss_function}.

\begin{table}[ht]
\centering
\caption{Comparison of the normalized quantum-state losses $\mathcal{L}^{Y}_{\rm norm}(\bmm)$ and $\mathcal{L}^{u}_{\rm norm}(\bmm)$ with the physical loss $\mathcal{L}_{\rm phys}(\bmm)$ reconstructed using a quantum algorithm. The Hadamard test probability $p_H^u(\bmm)$ and projected success probability $p_{\rm succ}(\bmm)$ are defined in \cref{eq:project_succ,eq:u_hada}, respectively, while their counterparts in the Carleman system, $p_H^Y(\bmm)$ and $p_{\rm succ}^Y(\bmm)$ are defined in \cref{eq:carleman_succ,eq:y_hada}.}
\label{tab:loss_function}
\resizebox{\linewidth}{!}{\begin{tabular}{cccc}
    \toprule 
    & Carleman $\mathcal{L}^{Y}_{\rm norm}(\bmm)$ & Normalized $\mathcal{L}^{u}_{\rm norm}(\bmm)$ & Physical $\mathcal{L}_{\rm phys}(\bmm)$ \\
    \midrule 
    Loss expression   
    & $\left\|\frac{\bY_T(\bmm)}{\|\bY_T(\bmm)\|}-\frac{\bY_{\rm obs}}{\|\bY_{\rm obs}\|}\right\|^2$  
    & $\left\|\frac{\bu_T(\bmm)}{\|\bu_T(\bmm)\|}-\frac{\bu_{\rm obs}}{\|\bu_{\rm obs}\|}\right\|^2$ 
    & $\|\bu_T(\bmm)-\bu_{\rm obs}\|^2$ \\ [1.5ex]

    Measurements & $p^Y_H(\bmm),\, p^Y_{\rm succ}(\bmm)$  & $p^u_H(\bmm),\, p_{\rm succ}(\bmm)$  & $p^u_H(\bmm),\, p_{\rm succ}(\bmm)$ \\ [1.2ex]

    \begin{tabular}[c]{@{}c@{}} Cost requirement \\ for $|\widetilde{\mathcal L}-\widehat{\mathcal{L}}|\leq\epsilon$ 
    \end{tabular}  
    & $\displaystyle \mathcal{O}\!\left(\frac{1}{\epsilon^2 p^Y_{\rm succ}(\bmm)}\right)$  
    & $\displaystyle \mathcal{O}\!\left(\frac{1}{\epsilon^2 p_{\rm succ}(\bmm)}\right)$  
    & $ \mathcal{O}\!\left( \frac{ \lambda(\bmm)^4\|\bu_0\|^4
    + \lambda(\bmm)^2 \|\bu_0\|^2 \|\bu_{\rm obs}\|^2} {\epsilon^2} \right)$ \\
    \bottomrule 
\end{tabular}}
\end{table}

Note that reconstructing $\mathcal{L}_{\rm phys}(\bmm)$ additionally requires the solver-dependent scaling factor $\lambda(\bmm)$ to be known classically. This condition is satisfied by terminal-state solvers based on block encodings and linear combinations of unitaries, which have explicit sub-normalization factors \cite{an2023linear,gutierrez2026quantum,berry2017quantum}. For QLSA-based history-state solvers whose normalization depends on an unknown history-state norm, the normalized losses remain directly accessible, whereas reconstruction of the physical loss requires an additional norm-estimation procedure \cite{berry2017quantum,liu2021efficient,MoralesPiraSchleich2025}.

\section{Error and complexity analysis} \label{sec:error_complex}
For a fixed candidate parameter $\bmm$, the error in the loss evaluation under quantum measurements consists of two components: the deterministic approximation error introduced by $U_{\rm solver}(\bmm)$ and the statistical error arising from finite quantum measurements. We first analyze how the deterministic error propagates to the loss function, and then quantify the measurement complexity required to control the statistical error in a single quantum loss evaluation.

\subsection{Approximation error of forward PDE solver}
Based on the notations defined in \cref{sec:quantum}, we analyze how the terminal-state approximation error $\epsilon_{\rm app}$ propagates to the normalized loss function and the physical loss function. A separate error analysis for the lifted Carleman normalized loss is omitted, since it is defined in the truncated Carleman space and serves only as an auxiliary objective. We first give the error analysis for normalized loss $\widetilde{\mathcal L}_{\rm norm}^{u}(\bmm)$.
\begin{lemma}[Error bound for the normalized loss] \label{lem:loss_norm_error}
Consider a nonzero observation vector $\bu_{\rm obs}$, the nonzero exact terminal solution $\bu_T(\bmm)$, and a nonzero approximation $\widetilde{\bu}_T(\bmm)$ satisfying $\|\bu_T(\bmm)-\widetilde{\bu}_T(\bmm)\|\leq\epsilon_{\rm app}$. Assuming that
\[
\|\bu_T(\bmm)-\bu_{\rm obs}\|\leq\delta\|\bu_{\rm obs}\|,
\]
where $0<\delta<1$. For the exact normalized loss $\mathcal L_{\rm norm}^{u}(\bmm)$ and its infinite-shot quantum approximation $\widetilde{\mathcal L}_{\rm norm}^{u}(\bmm)$ defined in \cref{eq:loss_norm_u}, the deterministic loss error satisfies
\[
    \left|\mathcal L_{\rm norm}^{u}(\bmm)-\widetilde{\mathcal L}_{\rm norm}^{u}(\bmm)\right|
    \leq \frac{4\epsilon_{\rm app}}{(1-\delta)\|\bu_{\rm obs}\|}.
\]
\end{lemma}

\begin{proof}[Proof of \cref{lem:loss_norm_error}]
The closeness assumption and the reverse triangle inequality imply
\[
    \|\bu_T(\bmm)\|\geq\|\bu_{\rm obs}\|-\|\bu_T(\bmm)-\bu_{\rm obs}\|
    \geq(1-\delta)\|\bu_{\rm obs}\|>0.
\]
Thus, $\bu_T(\bmm)\neq0$. Since $\widetilde{\bu}_T(\bmm)\neq0$ by assumption, the normalized terminal vectors are well defined. The difference between the normalized terminal states satisfies
\begin{align*}
    \left\|\frac{\bu_T(\bmm)}{\|\bu_T(\bmm)\|}-\frac{\widetilde{\bu}_T(\bmm)}{\|\widetilde{\bu}_T(\bmm)\|}\right\|
    &\leq\frac{\|\bu_T(\bmm)-\widetilde{\bu}_T(\bmm)\|}{\|\bu_T(\bmm)\|}
    +\frac{\big|\|\widetilde{\bu}_T(\bmm)\|-\|\bu_T(\bmm)\|\big|}{\|\bu_T(\bmm)\|} \\
    &\leq\frac{2\epsilon_{\rm app}}{(1-\delta)\|\bu_{\rm obs}\|}.
\end{align*}
    Using the definitions of $\mathcal{L}_{\rm norm}^u(\bmm)$ and $\widetilde{\mathcal{L}}_{\rm norm}^u(\bmm)$ together with the Cauchy--Schwarz inequality, we obtain
\begin{align*}
    \left|\mathcal{L}_{\rm norm}^u(\bmm)-\widetilde{\mathcal{L}}_{\rm norm}^u(\bmm)\right|
    \leq2\left\|\frac{\bu_T(\bmm)}{\|\bu_T(\bmm)\|}-\frac{\widetilde{\bu}_T(\bmm)}{\|\widetilde{\bu}_T(\bmm)\|}\right\| 
    \leq\frac{4\epsilon_{\rm app}}{(1-\delta)\|\bu_{\rm obs}\|}.
\end{align*}
\end{proof}

The corresponding error bound for the physical loss is given as follows.
\begin{lemma}[Error bound for the physical loss] \label{lem:loss_phys_error}
Consider the exact terminal solution $\bu_T(\bmm)$ and its approximation $\widetilde{\bu}_T(\bmm)$ satisfying $\|\bu_T(\bmm)-\widetilde{\bu}_T(\bmm)\|\leq\epsilon_{\rm app}$. For the exact physical loss $\mathcal L_{\rm phys}(\bmm)$ and its infinite-shot quantum approximation $\widetilde{\mathcal L}_{\rm phys}(\bmm)$ defined in \cref{eq:loss_phy}, the deterministic loss error satisfies
\[
    \left|\mathcal L_{\rm phys}(\bmm)-\widetilde{\mathcal L}_{\rm phys}(\bmm)\right|
    \leq 2\|\bu_T(\bmm)-\bu_{\rm obs}\|\epsilon_{\rm app}+\epsilon_{\rm app}^2.
\]
If the closeness condition $\|\bu_T(\bmm)-\bu_{\rm obs}\|\leq\delta\|\bu_{\rm obs}\|$ with $0<\delta<1$ in \cref{lem:loss_norm_error} further holds, then the deterministic loss error satisfies
\begin{equation} \label{phys_loss_error}
    \left|\mathcal L_{\rm phys}(\bmm)-\widetilde{\mathcal L}_{\rm phys}(\bmm)\right|
    \leq 2\delta\|\bu_{\rm obs}\|\epsilon_{\rm app}+\epsilon_{\rm app}^2.
\end{equation}
\end{lemma}

\begin{proof}[Proof of \cref{lem:loss_phys_error}]
By the reverse triangle inequality,
\[
\big|\|\bu_T(\bmm)-\bu_{\mathrm{obs}}\|-\|\widetilde{\bu}_T(\bmm)-\bu_{\mathrm{obs}}\|\big|\leq\|\bu_T(\bmm)-\widetilde{\bu}_T(\bmm)\|\leq\epsilon_{\mathrm{app}}.
\]
The triangle inequality also gives
\[
\|\widetilde{\bu}_T(\bmm)-\bu_{\mathrm{obs}}\|\leq\|\bu_T(\bmm)-\bu_{\mathrm{obs}}\|+\epsilon_{\mathrm{app}}.
\]
Consequently,
\begin{align*}
\left|\mathcal{L}_{\mathrm{phys}}(\bmm)-\widetilde{\mathcal{L}}_{\mathrm{phys}}(\bmm)\right|
&=\left|\|\bu_T(\bmm)-\bu_{\mathrm{obs}}\|^2-\|\widetilde{\bu}_T(\bmm)-\bu_{\mathrm{obs}}\|^2\right| \\
&\leq\epsilon_{\mathrm{app}}\left(\|\bu_T(\bmm)-\bu_{\mathrm{obs}}\|+\|\widetilde{\bu}_T(\bmm)-\bu_{\mathrm{obs}}\|\right) \\
&\leq2\|\bu_T(\bmm)-\bu_{\mathrm{obs}}\|\epsilon_{\mathrm{app}}+\epsilon_{\mathrm{app}}^2.
\end{align*}
Under the assumption of $\|\bu_T(\bmm)-\bu_{\mathrm{obs}}\|\leq\delta\|\bu_{\mathrm{obs}}\|$, \Cref{phys_loss_error} follows.
\end{proof}

\subsection{Complexity analysis for loss evaluation}
From the preceding analysis, \cref{tab:loss_function} summarizes the leading sampling complexity required to achieve a prescribed additive loss accuracy. For one BayesOpt loss evaluation, we further characterize this sampling complexity from a statistical perspective by providing a high-probability bound for the finite-shot loss estimation.

\begin{lemma}[Statistical error for normalized loss estimation]
\label{lem:stat_error}
Let $\widetilde{\mathcal L}^u_{\rm norm}(\bmm)$ be the infinite-shot normalized loss and $\widehat{\mathcal L}^u_{\rm norm}(\bmm)$ be its finite-shot estimator obtained from independent circuit repetitions, respectively. Suppose that Hadamard test probability $p_H^u(\bmm) > 0$ and projected success probability  $p_{\rm succ}(\bmm)>0$, then for any prescribed failure probability $\rho\in(0,1)$ and statistical accuracy $0<\epsilon_{\rm stat}<1$, the estimator satisfies
\[
    \mathbb P\left(\left|\widehat{\mathcal L}^u_{\rm norm}(\bmm)-\widetilde{\mathcal L}^u_{\rm norm}(\bmm)\right|\leq\epsilon_{\rm stat}\right)\geq 1-\rho.
\]

Each Hadamard test circuit queries both $O_{\rm prep}$ and $O_{\rm obs}$ once, while each circuit for estimating $p_{\rm succ}(\bmm)$ queries $O_{\rm prep}$ once. The numbers of queries to $O_{\rm prep}$ and $O_{\rm obs}$ required to achieve the prescribed statistical accuracy both satisfy
\[
    \mathcal O\left(\frac{\lambda(\bmm)^2\left\|\bu_0\right\|^2}{ \left\|\widetilde{\bu}_T(\bmm)\right\|^2\epsilon_{\rm stat}^2}\log\frac{1}{\rho}\right),
\]
where the scaling factor $\lambda(\bmm)>0$ depends on the quantum PDE solver.
\end{lemma}

\begin{proof}[Proof of \cref{lem:stat_error}]
Let $\widehat{p}_H^u(\bmm)$ and $\widehat{p}_{\rm succ}(\bmm)$ denote the empirical frequencies obtained from $N_H$ Hadamard test and $N_q$ projected-success measurements, respectively. The finite-shot estimator and the corresponding infinite-shot loss are 
\[
\widehat{\mathcal{L}}_{\mathrm{norm}}^u(\bmm)=2-\frac{4\widehat{p}_H^u(\bmm)-2}{\sqrt{\widehat{p}_{\rm succ}(\bmm)}}, \quad 
\widetilde{\mathcal{L}}_{\mathrm{norm}}^u=2(\bmm)-\frac{4p_H^u(\bmm)-2}{\sqrt{p_{\rm succ}(\bmm)}}.
\]

Since the observation state is normalized and entirely supported on the projected-success subspace, the Cauchy--Schwarz inequality gives
\[
|2p_H^u(\bmm)-1| = \left|\operatorname{Re}\langle\psi_{\rm obs}|\psi_T(\bmm)\rangle\right|
\leq \left|\langle\psi_{\rm obs}|\psi_T(\bmm)\rangle\right| \leq \|P_u|\psi_T(\bmm)\rangle\| = \sqrt{p_{\rm succ}(\bmm)}.
\]

For $0<\epsilon_{\mathrm{stat}}\leq1$, consider the simultaneous concentration bounds
\[
\big|\widehat{p}_H^u(\bmm)-p_H^u(\bmm)\big|\leq\frac{\epsilon_{\mathrm{stat}}\sqrt{p_{\rm succ}(\bmm)}}{16},
\qquad
|\widehat{p}_{\rm succ}(\bmm)-p_{\rm succ}(\bmm)|\leq\frac{\epsilon_{\mathrm{stat}}p_{\rm succ}(\bmm)}{8}.
\]
On this event,
$
\widehat{p}_{\rm succ}(\bmm)\geq\left(1-\frac{\epsilon_{\mathrm{stat}}}{8}\right)p_{\rm succ}(\bmm)\geq\frac{7p_{\rm succ}(\bmm)}{8}>0,
$
so the finite-shot estimator is well defined. Moreover,
\begin{align*}
\big|\widehat{\mathcal{L}}_{\mathrm{norm}}^u(\bmm)
-\widetilde{\mathcal{L}}_{\mathrm{norm}}^u(\bmm)\big|
&\leq
\frac{4\big|\widehat{p}_H^u(\bmm)-p_H^u(\bmm)\big|}
{\sqrt{\widehat{p}_{\rm succ}(\bmm)}}
\\
&\quad
+|4p_H^u(\bmm)-2|
\left|
\frac{1}{\sqrt{\widehat{p}_{\rm succ}(\bmm)}}
-\frac{1}{\sqrt{p_{\rm succ}(\bmm)}}
\right|
\\
&\leq
\frac{\epsilon_{\mathrm{stat}}}{4}\sqrt{\frac{8}{7}}
+\frac{2|\widehat{p}_{\rm succ}(\bmm)-p_{\rm succ}(\bmm)|}
{\sqrt{\widehat{p}_{\rm succ}(\bmm)}
\big(\sqrt{\widehat{p}_{\rm succ}(\bmm)}+\sqrt{p_{\rm succ}(\bmm)}\big)}
\\
&\leq
\frac{\epsilon_{\mathrm{stat}}}{4}\sqrt{\frac{8}{7}}
+\frac{\epsilon_{\mathrm{stat}}}{4}\sqrt{\frac{8}{7}}
<\epsilon_{\mathrm{stat}}.
\end{align*}

Since $\widehat{p}_H^u(\bmm)$ is the empirical mean of $N_H$ independent Bernoulli measurements, Hoeffding's inequality yields
\[
\mathbb{P}\bigg(\big|\widehat{p}_H^u(\bmm)-p_H^u(\bmm)\big|>\frac{\epsilon_{\mathrm{stat}}\sqrt{p_{\rm succ}(\bmm)}}{16}\bigg)\leq2\exp\left(-\frac{N_H p_{\rm succ}(\bmm) \epsilon_{\mathrm{stat}}^2}{128}\right).
\]
Similarly, the multiplicative Chernoff bound gives
\[
\mathbb{P}\bigg(|\widehat{p}_{\rm succ}(\bmm)-p_{\rm succ}(\bmm)|>\frac{\epsilon_{\mathrm{stat}}p_{\rm succ}(\bmm)}{8}\bigg)\leq2\exp\left(-\frac{N_q p_{\rm succ}(\bmm) \epsilon_{\mathrm{stat}}^2}{192}\right).
\]
Therefore, it is sufficient to choose
$
N_H\geq\frac{128\log(4/\rho)}{p_{\rm succ}(\bmm)\epsilon_{\mathrm{stat}}^2},
N_q\geq\frac{192\log(4/\rho)}{p_{\rm succ}(\bmm)\epsilon_{\mathrm{stat}}^2}.
$
By the union bound, the two concentration inequalities then hold simultaneously with probability at least $1-\rho$. Consequently,
\[
\mathbb{P}\big(\big|\widehat{\mathcal{L}}_{\mathrm{norm}}^u(\bmm)-\widetilde{\mathcal{L}}_{\mathrm{norm}}^u(\bmm)\big|\leq\epsilon_{\mathrm{stat}}\big)\geq1-\rho.
\]
Therefore, $N_H,N_q=\mathcal{O}\left(\frac{1}{p_{\rm succ}(\bmm)\epsilon_{\mathrm{stat}}^2}\log\frac{1}{\rho}\right)$ and  $N_{\mathrm{total}}=\mathcal{O}\left(\frac{1}{p_{\rm succ}(\bmm)\epsilon_{\mathrm{stat}}^2}\log\frac{1}{\rho}\right).$

Under the one-query state-preparation assumption, each projected-success measurement queries
$O_{\rm prep}$ once, whereas each Hadamard test measurement queries both $O_{\rm prep}$ and $O_{\rm obs}$ once. Therefore, the numbers of queries to the two oracles both scale as
$
    \mathcal{O}\left(\frac{1}{p_{\rm succ}(\bmm)\epsilon_{\mathrm{stat}}^2}\log\frac{1}{\rho}\right).
$
The gate complexity of each circuit repetition depends on the specific implementation of $U_{\rm solver}(\bmm)$ used to prepare the terminal state.
\end{proof}

\begin{table}[ht]
\centering
\caption{Oracle-query complexity for one BayesOpt loss evaluation under given failure probability $\rho$ and statistical accuracy $\epsilon_{\textup{stat}}$.}
\label{tab:complexity_loss}
\footnotesize
\renewcommand{\arraystretch}{1.5}
\setlength{\tabcolsep}{4pt}
\begin{tabular}{ccc}
\toprule
Loss function & Queries to $O_{\rm prep}$ & Queries to $O_{\rm obs}$ \\
\midrule
$\mathcal{L}_{\rm norm}^{Y}(\bmm)$ & $\displaystyle \mathcal{O}\!\left(\frac{1}{p^Y_{\rm succ}(\bmm)\epsilon_{\rm stat}^{2}}\log\frac{1}{\rho}\right)$ & $\displaystyle \mathcal{O}\!\left(\frac{1}{p^Y_{\rm succ}(\bmm)\epsilon_{\rm stat}^{2}}\log\frac{1}{\rho}\right)$ \\[1.7ex]
$\mathcal{L}_{\rm norm}^{u}(\bmm)$ & $\displaystyle \mathcal{O}\!\left(\frac{1}{p_{\rm succ}(\bmm)\epsilon_{\rm stat}^{2}}\log\frac{1}{\rho}\right)$ & $\displaystyle \mathcal{O}\!\left(\frac{1}{p_{\rm succ}(\bmm)\epsilon_{\rm stat}^{2}}\log\frac{1}{\rho}\right)$ \\[1.7ex]
$\mathcal{L}_{\rm phys}(\bmm)$ & $\displaystyle \mathcal{O}\!\left(\frac{\lambda(\bmm)^{4}\|\bu_0\|^{4}+\lambda(\bmm)^{2}\|\bu_0\|^{2}\|\bu_{\rm obs}\|^{2}}{\epsilon_{\rm stat}^{2}}\log\frac{1}{\rho}\right)$ & $\displaystyle \mathcal{O}\!\left(\frac{\lambda(\bmm)^{2}\|\bu_0\|^{2}\|\bu_{\rm obs}\|^{2}}{\epsilon_{\rm stat}^{2}}\log\frac{1}{\rho}\right)$ \\
\bottomrule
\end{tabular}
\end{table}

The same argument applies to the lifted normalized loss by replacing $p_H^u(\bmm)$ and $p_{\rm succ}$ with $p_H^Y(\bmm)$ and $p^Y_{\rm succ}$, respectively. For the physical loss, the stated high-probability bound follows directly from applying Hoeffding's inequality to the linear estimator in \cref{eq:loss_phy} and using a union bound. Under the one-query state-preparation assumption, each projected-success measurement queries $O_{\rm prep}$ once, whereas each Hadamard test measurement queries both $O_{\rm prep}$ and $O_{\rm obs}$ once. The resulting oracle-query complexities for one BayesOpt loss evaluation are summarized in \cref{tab:complexity_loss}. The detailed gate complexity of each circuit depends on the specific implementation of $U_{\rm solver}(\bmm)$.

\section{Numerical experiments} \label{sec:numerical}
In this section, we investigate the numerical performance of the proposed loss functions within the quantum-assisted Bayesian optimization framework. 
All computations were performed on a laptop  with an Intel Core i9-13900HX processor. To assess the accuracy, we report the relative $L^2$ error ${\|\bu_{\rm obs} - \bu_T(\bmm)\|}/{\|\bu_{\rm obs}\|}$ as the evaluation metric, where $\bu_{\rm obs}$ denotes the noise-free observation, and $\bu_T(\bmm)$ denotes the numerical solution at the estimated optimal parameter $\bmm$.

\subsection{1-D convection-diffusion equation}
We first focus on a linear problem to assess the accuracy of the Taylor-truncated approximation to the forward solution and investigate its impact on the resulting parameter inversion. 

Consider the 1-D convection-diffusion equation
\[
    \begin{cases}
    \partial_tu = \left( r - \frac{1}{2}\sigma^2 \right) \partial_xu + \frac{1}{2}\sigma^2 \partial_{xx}u, \\
    u(x,0)=e^{-x^2}.
    \end{cases} \quad (x,t)\in[-L_0,L_0]\times[0,T],
\]
with homogeneous Dirichlet boundary conditions $u(-L_0,t)=u(L_0,t)=0$, where $L_0=4$ and $T=2$. We set $\bmm=(r, \sigma^2)$ and examine the approximation accuracy of the Taylor-truncated forward solution for different truncation orders. \cref{fig:cdeq_taylor} exhibits the numerical results and the corresponding experimental settings.

\begin{figure}[ht]
    \centering
    \includegraphics[width=\linewidth,height=0.25\linewidth]{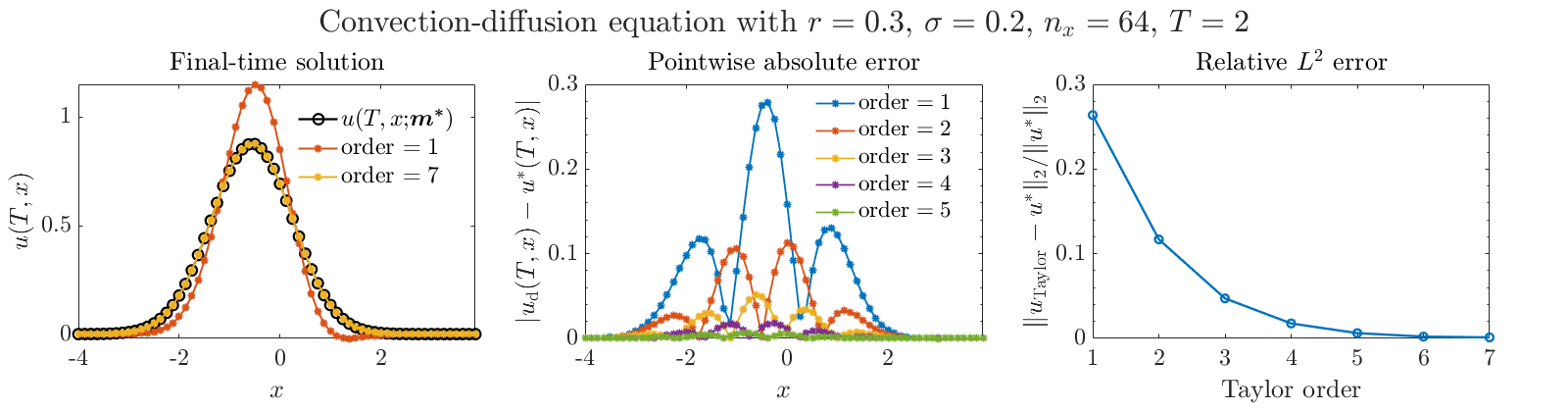}%
    \caption{Integration of the 1-D convection-diffusion equation using a Taylor-truncated approximation.}
    \label{fig:cdeq_taylor}
\end{figure}

We set the truncation order $n_{\rm Tay}=5$ and use ${n_x}=16$ spatial discretization points, with $\bmm^*=(r^*,\sigma^{*2})=(0.3,0.04)$ as the true parameter. To generate the training data, we independently perturb $r$ and $\sigma$ and sample 30 parameter points according to $\bmm^{\mathrm{train}}_i=\big(r^*(1+0.3z_{i,1}),\left[\sigma^*(1+0.3z_{i,2})\right]^2\big), \bm{z}_i\sim\mathcal{N}(\mathbf{0},I_2)$. The corresponding noisy numerical solutions are generated using the exact forward solver as
\[
    {\bu}_{\mathrm{num}}(T;\bmm^{\mathrm{train}}_i)=\bu(T;\bmm^{\mathrm{train}}_i)+0.2\cdot\frac{\|\bu(T;\bmm^*)\|}{\sqrt{n_x}}\cdot\boldsymbol{\xi}_i,\quad \boldsymbol{\xi}_i\sim\mathcal{N}(\mathbf{0},I_{n_x}).
\]

To further validate the proposed framework, we simulate the quantum circuits using the PennyLane Python library\footnote{PennyLane is an open-source software framework for quantum computing, quantum machine learning, and quantum chemistry; see \url{https://pennylane.ai}.} \cite{bergholm2018pennylane}. For the numerical experiment, the terminal state is prepared using generalized quantum signal processing (GQSP) applied to an $n_{\rm Tay}$-regular block encoding framework \cite{gutierrez2026quantum}. The corresponding required GQSP processing angles can be computed numerically within the same software environment. The block-encoding and the controlled unitary required by the Hadamard test are first constructed classically and then implemented as quantum circuit operators in the simulator. 

Using $\mathcal{L}_{\rm phys}$ as the loss function, we set the annealing parameter to be $\gamma=1/100$ and perform $100$ BayesOpt iterations in all experiments. \cref{fig:cdeq_classical} presents the classical results as a baseline. Since finite-shot Hadamard test measurements with sample size $N_H$ may produce a negative estimate of the physical loss, we use the clipped estimator $\widehat{\mathcal{L}}_{\rm phys}^{\rm clip}:=\max\{0,\widehat{\mathcal{L}}_{\rm phys}\}$ at each evaluation. This clipping may result in multiple evaluated parameters attaining the same objective value $f(\bmm)=1$ as shown in \cref{fig:cdeq_quantum}, making the optimized best parameter non-unique. The GPR surrogate provides a natural way to resolve this ambiguity: We select the parameter that maximizes the posterior mean of the objective function as the predicted inversion result and denote it by $\bmm_{\rm opt} = \arg\max_{\bmm} \mathbb{E}[f(\bmm) \mid \mathcal{D}]$. 

\begin{figure}[ht]
    \centering
    \subfloat[Baseline by classical computing with $\bmm_{\rm opt}=(0.2993,0.0343)$.]
    {\label{fig:cdeq_classical}\includegraphics[width=\linewidth,height=0.3\linewidth]{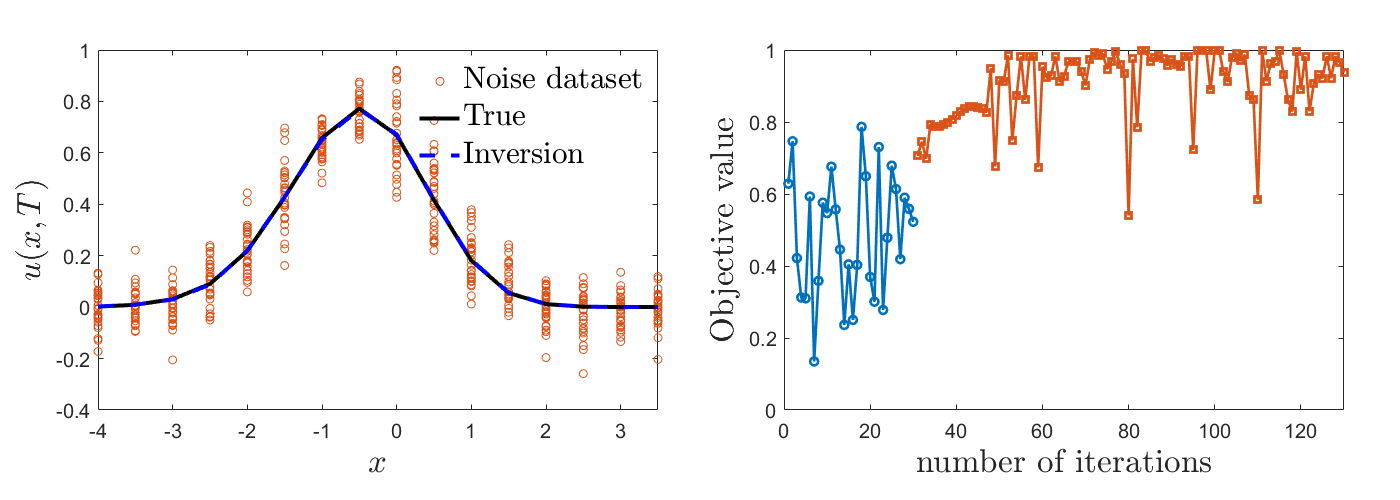}}\\
    \subfloat[One run with $N_H=10^9$ Hadamard test measurements and $\bmm_{\rm opt}=(0.2913,0.0596)$.]
    {\label{fig:cdeq_quantum}\includegraphics[width=\linewidth,height=0.3\linewidth]{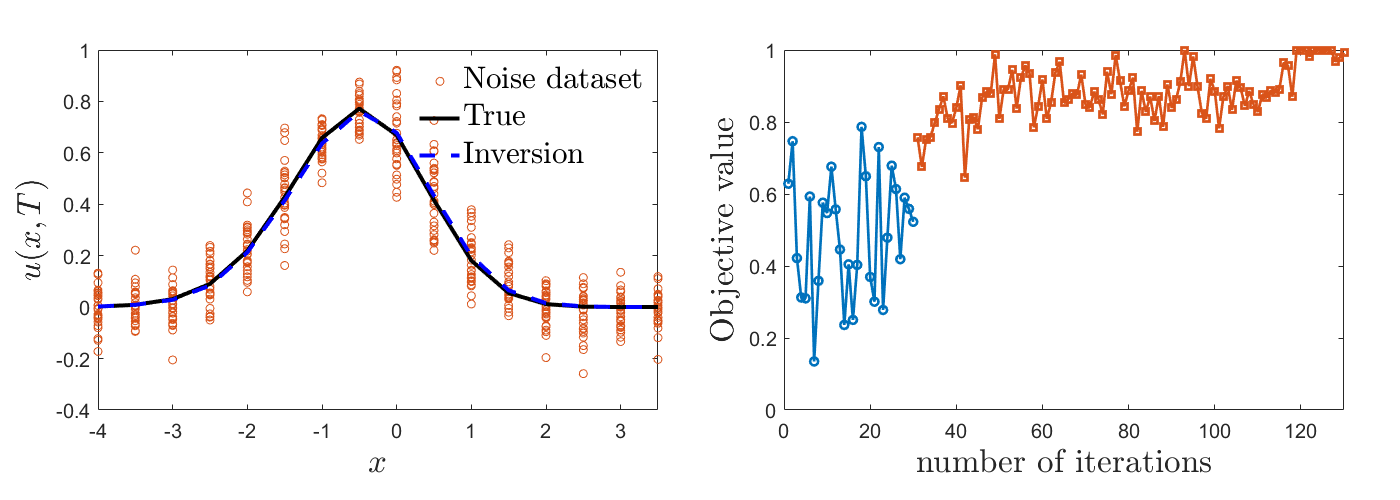}} \\
    \caption{Representative results for the 1-D convection-diffusion equation with $n_{\rm Tay}=5$ and $\bmm^*=(0.3,0.04)$. (a) baseline by classical computing with $\bmm_{\rm opt}=(0.2993,0.0343)$. (b) one run with $N_H=10^9$ Hadamard test measurements, yielding $\bmm_{\rm opt}=(0.2913,0.0596)$.}
    \label{fig:cdeq_BO}
\end{figure}

\begin{table}[ht]
\centering
\footnotesize
\setlength{\tabcolsep}{6pt}
\renewcommand{\arraystretch}{1.12}
\caption{Numerical results for the 1-D convection-diffusion equation with $\bmm^*=(0.3,0.04)$. For quantum simulation, results are reported as the mean $\pm$ standard deviation over 10 independent runs.}
\label{tab:cdeq_BO}
\begin{tabular}{@{}ccccc@{}}
\toprule
\multirow{2}{*}{Method}
& \multirow{2}{*}{\shortstack{Expected number\\of $\widehat{\mathcal{L}}_{\rm phys} \le 0$}} 
& \multicolumn{2}{c}{Estimated $\bmm_{\rm opt}$} & \multirow{2}{*}{\shortstack{Relative $L^2$ error}} \\ \cmidrule(lr){3-4}
& & $r$ & $\sigma^2$ & \\
\midrule
 Classical     & $0$ & $0.2993$ & $0.0343$ & $6.62\times10^{-3}$ \\
$N_H=\infty$   & $0$ & $0.2970$ & $0.0388$ & $5.74\times10^{-3}$ \\ 
 $N_H=10^3$ & $51.6\pm5.9$ & $0.4024\pm0.2116$ & $0.0762\pm0.1008$ & $(3.98\pm9.77)\times10^{-1}$ \\
 $N_H=10^5$ & $50.2\pm4.9$ & $0.3284\pm0.0158$ & $0.0623\pm0.0560$ & $(6.37\pm3.51)\times10^{-2}$ \\
 $N_H=10^7$ & $43.3\pm4.9$ & $0.3370\pm0.0072$ & $0.1314\pm0.0283$ & $(8.27\pm2.04)\times10^{-2}$ \\
 $N_H=10^9$ & $7.4\pm8.0$ & $0.3039\pm0.0324$ & $0.0589\pm0.0536$ & $(4.06\pm3.74)\times10^{-2}$ \\
 $N_H=10^{11}$ & $0.9\pm1.7$ & $0.2976\pm0.0169$ & $0.0358\pm0.0343$ & $(3.72\pm1.51)\times10^{-2}$ \\
\bottomrule
\end{tabular}
\end{table}

The detailed inversion results and the corresponding relative errors are reported in \cref{tab:cdeq_BO}, while the results obtained from quantum simulators with finite-shot Hadamard test measurements are averaged over 10 independent experimental runs to assess the stability of the proposed framework. Here, we treat the success probability $p_{\rm succ}$ as exact and consider only the finite-shot statistical error arising from the Hadamard test. The quantum-circuit simulations yield reasonable parameter estimates, demonstrating the feasibility of the proposed framework under finite-shot measurement noise.

\subsection{2-D convection-diffusion equation}
We next consider the 2-D case of the convection-diffusion equation defined in $(x,y,t)\in\Omega\times[0,T],$
\[ \begin{cases}
    \partial_t u  = D\left(\partial_{xx}u+\partial_{yy}u\right)  -  v\left(\partial_x u
    +\partial_y u\right),  \\[1mm]
    u(x,y,0) = \exp\left( -\frac{(x-0.25)^2+(y-0.35)^2}{2(0.06)^2} \right)  +
    0.7\exp\left( -\frac{(x-0.70)^2+(y-0.20)^2}{2(0.09)^2} \right),
    \end{cases}
\]
with periodic boundaries in both spatial directions, where $T=1$ and $\Omega=[0,1]^2$. We set \(\bmm=(D,v)\) and assume that the convection velocities in the two spatial directions are identical. The initial condition is periodized in both spatial variables.

We use a $64\times64$ spatial grid with $n_x=n_y=64$ and set $\bmm^*=(D^*,v^*)=(0.05,1)$ as the true parameter. The parameter bounds are $D\in[0,0.2]$ and $v\in[0.2,1.8]$. To generate the initial training data, we independently perturb $D$ and $v$ and sample 30 parameter points
according to $\bmm^{\mathrm{train}}_i = \left( D^*(1+0.3z_{i,1}), v^*(1+0.3z_{i,2}) \right)$,
where $\bm{z}_i\sim\mathcal{N}(\mathbf{0},I_2)$. Samples outside the prescribed parameter bounds are clipped to the admissible parameter domain. For a prescribed relative noise level $\eta$, the noisy numerical solution at each evaluated parameter is generated as
\[
    \bu_{\mathrm{num}}(T;\bmm_i) = \bu(T;\bmm_i) + \eta\cdot \frac{\|\bu(T;\bmm^*)\|}{\sqrt{n_x n_y}} \cdot \boldsymbol{\xi}_i,
    \qquad \boldsymbol{\xi}_i  \sim  \mathcal{N}(\mathbf{0},I_{n_x\cdot n_y}).
\]

We consider the relative noise levels $\eta\in\{0.1,0.2,0.3\}$ to investigate the robustness of the proposed framework. These values specify the noise magnitude used in both the
initial training data and the subsequent BayesOpt evaluations. \cref{tab:cdeq_2d_noise} reports the inferred parameters and relative $L^2$ errors for three noise levels. \cref{fig:cdeq_2d_inverse} compares the exact solution, the reconstructed solution obtained using the mean GPR prediction $\overline{\bmm}_{\mathrm{opt}}=(0.0530,1.0000)$ estimate at $\eta=0.2$ over 10 independent runs, and their pointwise absolute error. The results demonstrate that the proposed framework provides stable parameter estimates under the prescribed noise perturbations.

\begin{table}[ht]
\centering
\footnotesize
\setlength{\tabcolsep}{6pt}
\renewcommand{\arraystretch}{1.12}
    \caption{Numerical results for the 2-D convection-diffusion equation with $\bmm^*=(0.05,1)$. Results are reported as the mean $\pm$ standard deviation over 10 independent runs.}
\label{tab:cdeq_2d_noise}
\begin{tabular}{@{}cccc@{}}
    \toprule
    \multirow{2}{*}{Noise level}
    & \multicolumn{2}{c}{Estimated $\bmm_{\rm opt}$}
    & \multirow{2}{*}{\shortstack{Relative $L^2$ error}} \\
    \cmidrule(lr){2-3} & $D$ & $v$ & \\
    \midrule
    $\eta=0.1$ & $0.0511\pm0.0010$ & $0.9975\pm0.0131$ & $(1.48\pm0.40)\times10^{-2}$ \\
    $\eta=0.2$  & $0.0530\pm0.0021$ & $1.0000\pm0.0134$ & $(2.26\pm0.90)\times10^{-2}$ \\
    $\eta=0.3$  & $0.0568\pm0.0018$ & $0.9949\pm0.0123$  & $(3.94\pm0.82)\times10^{-2}$ \\
    \bottomrule
\end{tabular}
\end{table}

\begin{figure}[H]
    \centering
    \includegraphics[width=\linewidth,height=0.25\linewidth]{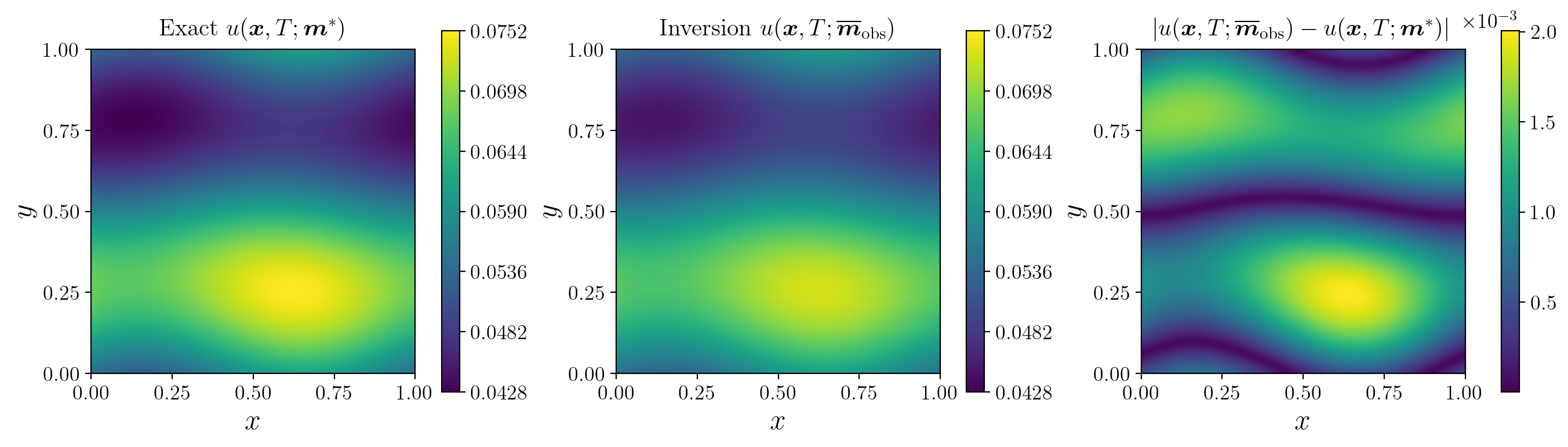}%
    \caption{Comparison of the reference and reconstructed solutions for the 2-D convection-diffusion equation at $T=1$ and noise level $\eta=0.2$. Left: exact solution with $\bmm^*=(0.05,1)$. Middle: reconstructed solution using the mean GPR estimate $\overline{\bmm}_{\mathrm{opt}}=(0.0530,1.0000)$ over 10 independent runs. Right: pointwise absolute error between the reconstructed and exact solutions.}
    \label{fig:cdeq_2d_inverse}
\end{figure}

\subsection{1-D forced viscous Burgers equation} \label{sec:burgers_eq}
We next consider the 1-D forced viscous Burgers equation
\begin{equation}
    \partial_t u + u\partial_x u = \nu \partial^2_x u  + f, \quad (x,t)\in[-L_0/2,L_0/2]\times[0,T],
\end{equation}
where the initial condition $u(x,0)=U_0\sin(2\pi x/L_0)$ is imposed on the computational domain with $L_0=1$, together with homogeneous Dirichlet boundary conditions $u(-L_0/2,t)=u(L_0/2,t)=0$. We use a time-independent localized off-center Gaussian forcing term $f(x)=U_0\exp\left(-\frac{(x-L_0/4)^2}{2(L_0/32)^2}\right)$, where $U_0=1/\sqrt{n_x-1}$. The viscosity is parameterized as $\nu=U_0L_0/\mathrm{Re}$, where $\mathrm{Re}$ denotes the Reynolds number. In the inverse problem, $\bmm=\mathrm{Re}$ is treated as the unknown parameter over the admissible interval $\mathrm{Re}\in(0,30]$. 

\subsubsection{Forward simulation}

Three cases with different Reynolds numbers and final times are considered, as summarized in \cref{tab:Burger_case}. First, we use Case II to investigate the effect of the Carleman truncation order $N$ on the accuracy of the numerical solution\footnote{We use the open-source implementation of Carleman linearization for the Burgers equation, available at \url{https://github.com/hermankolden/CarlemanBurgers}.}. Since our focus here is solely on the error induced by the Carleman truncation, the matrix-exponential action in the forward solver is evaluated using a scaling strategy together with an $n_{\rm Tay}=2$ truncated Taylor approximation \cite{al2011computing}, so that the associated numerical error is negligible. \cref{fig:Carleman_linear} compares the exact terminal solution with the Carleman approximations at different truncation orders and reports the corresponding pointwise and relative $L^2$ errors, illustrating the improvement in approximation accuracy as $N$ increases.

\begin{table}[ht]
    \centering
    \footnotesize
    \caption{Numerical settings for the forced viscous Burgers equation.}
    \begin{tabular}{cccccc}
        \toprule
        Numerical setting & ${n_x}$ & $T$ & Carleman order $N$  & Reynolds number  \\
        \midrule
        Case I   & 16  & 1  & 3  & 18  \\
        Case II  & 16  & 2  & 3  & 14  \\
        Case III & 16  & 3  & 2  & 10  \\
        \bottomrule
    \end{tabular}
    \label{tab:Burger_case}
\end{table}

\begin{figure}[htbp]
    \label{fig:Carleman_linear}
    \centering
    \includegraphics[width=\linewidth,height=0.25\linewidth]{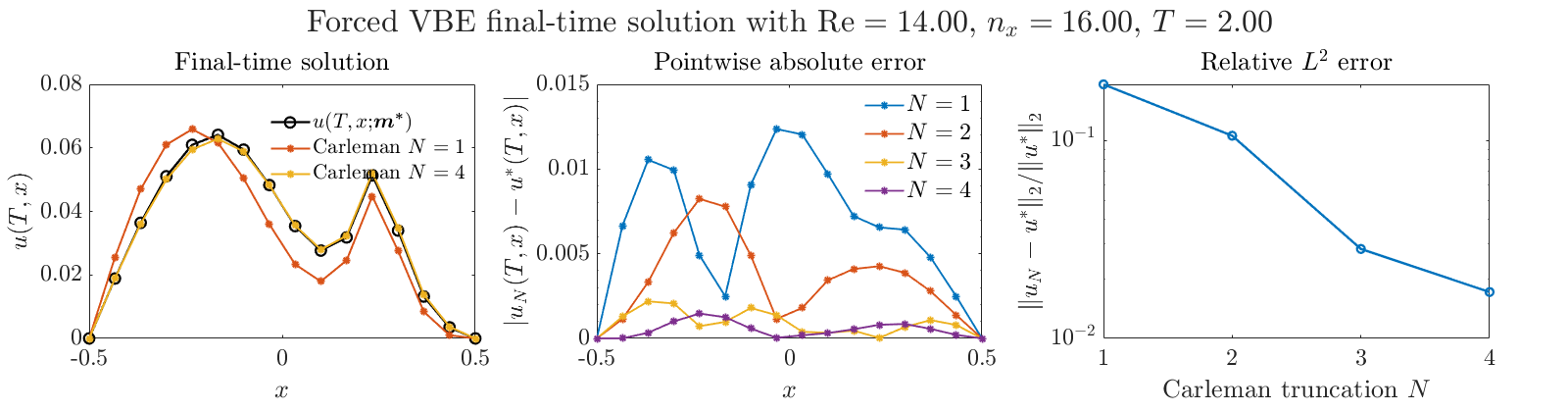}%
    \caption{Numerical integration of the forced viscous Burgers equation using Carleman linearization on a classical computer. Left: exact solution computed using MATLAB's ode45 solver. Center: pointwise absolute errors between Carleman solutions at different truncation levels $N$ and the exact solution. Right: relative $L^2$ errors versus the truncation order $N$.}
\end{figure}

\subsubsection{PDE inversion under different loss functions}

We next solve the inverse problem based on Carleman linearization and the truncated Taylor approximation, and then examine the inversion results for different terminal times $T$ and Reynolds numbers. All numerical experiments in this subsection are performed using classical simulations with a constant annealing parameter $\gamma=1/100$ and \(100\) BayesOpt iterations, except for all Case III experiments and the Case II experiments using the physical loss with $N_H=10^6$ and the normalized lifted loss with $N_H=10^4$, for which 150 BayesOpt iterations are employed. We consider the predicted best parameter $\bmm_{\rm opt}=\arg\max_{\bmm}\mathbb{E}[f(\bmm)\mid\mathcal{D}]$ inferred from the trained GPR.  

For data generation, we sample 30 training parameters according to $\bmm^{\mathrm{train}}_i = \bmm^*(1 + 0.3 z_i), z_i \sim \mathcal{N}(0,1).$ The corresponding solution snapshots are obtained using an exact forward solver. We introduce multiplicative noise in the forward evaluations used to construct the training dataset, 
\[
{\bu}_{\mathrm{num}}(T;\bmm^{\mathrm{train}}_i)= \bu_T(\bmm^{\mathrm{train}}_i) \odot \left(\mathbf{1} + 0.2\,\bm{\xi}_i\right), \bm{\xi}_i \sim \mathcal{N}(\mathbf{0}, I_{n_x}),
\]
so that the homogeneous boundary conditions are preserved.

\begin{table}[ht]
\centering
\footnotesize
\caption{Parameter-inversion results for the forced viscous Burgers equation. Finite-shot Hadamard test results are reported as the mean $\pm$ standard deviation over 10 independent runs.}

\resizebox{\linewidth}{!}{
\begin{tabular}{cccccc}
\toprule
\multirow{2}{*}{Case} & \multirow{2}{*}{True $\bmm^*$} & \multicolumn{2}{c}{Infinite-shot baseline: $N_H = \infty$} & \multicolumn{2}{c}{$N_H=10^2$} \\
\cmidrule(lr){3-4}\cmidrule(lr){5-6}
& & Estimated $\bmm_{\rm opt}$ & Relative $L^2$ error & Estimated $\bmm_{\rm opt}$ & Relative $L^2$ error \\
\midrule

\multicolumn{6}{c}{Physical loss $\mathcal{L}_{\textup{phys}}(\bmm) = \|\bu_T(\bmm)\|^2 + \|\bm{u}_{\rm obs}\|^2 - 2\mathrm{Re}\braket{\bm{u}_{\rm obs},\bu_T(\bmm)}$} \\
\midrule
I   & 18.00 & 17.24 & $5.02\times10^{-2}$ & $17.22\pm0.15$ & $5.06\times10^{-2}\pm4.95\times10^{-4}$ \\
II  & 14.00 & 13.87 & $2.91\times10^{-2}$ & $13.87\pm0.16$ & $3.08\times10^{-2}\pm3.19\times10^{-3}$ \\
III & 10.00 & 9.85  & $8.66\times10^{-3}$ & $9.85\pm0.28$  & $1.42\times10^{-2}\pm5.74\times10^{-3}$ \\

\midrule
\multicolumn{6}{c}{Normalized lifted loss $\mathcal{L}_{\rm norm}^{Y}(\bmm) = 2 - 2\frac{\mathrm{Re}\braket{\bY_{\rm obs},\bY_T(\bmm)}}{\|\bY_{\rm obs}\|\|\bY_T(\bmm)\|}$} \\
\midrule
I   & 18.00 & 17.24 & $5.02\times10^{-2}$ & $17.63\pm0.74$ & $5.89\times10^{-2}\pm8.81\times10^{-3}$ \\
II  & 14.00 & 13.90 & $2.87\times10^{-2}$ & $13.96\pm0.82$ & $5.13\times10^{-2}\pm3.25\times10^{-2}$ \\
III & 10.00 & 9.76  & $9.18\times10^{-3}$ & $12.85\pm4.81$ & $1.48\times10^{-1}\pm2.11\times10^{-1}$ \\

\midrule
\multicolumn{6}{c}{Normalized loss $\mathcal{L}_{\rm norm}^{u}(\bmm) = 2 - 2\frac{\mathrm{Re}\braket{\bu_{\rm obs},\bu_T(\bmm)}}{\|\bu_{\rm obs}\|\|\bu_T(\bmm)\|}$} \\
\midrule
I   & 18.00 & 17.81 & $5.48\times10^{-2}$ & $18.53\pm1.40$ & $8.26\times10^{-2}\pm2.81\times10^{-2}$ \\
II  & 14.00 & 14.02 & $2.85\times10^{-2}$ & $14.13\pm0.14$ & $3.13\times10^{-2}\pm3.89\times10^{-3}$ \\
III & 10.00 & 10.90 & $5.02\times10^{-2}$ & $5.54\pm2.17$  & $2.94\times10^{-1}\pm1.15\times10^{-1}$ \\
\bottomrule
\end{tabular}
}

\vspace{1em}
\resizebox{\linewidth}{!}{
\begin{tabular}{cccccc}
\toprule
\multirow{2}{*}{Case} & \multirow{2}{*}{True $\bmm^*$} & \multicolumn{2}{c}{$N_H=10^4$} & \multicolumn{2}{c}{$N_H=10^6$} \\
\cmidrule(lr){3-4}\cmidrule(lr){5-6}
& & Estimated $\bmm_{\rm opt}$ & Relative $L^2$ error & Estimated $\bmm_{\rm opt}$ & Relative $L^2$ error \\
\midrule

\multicolumn{6}{c}{Physical loss $\mathcal{L}_{\textup{phys}}(\bmm) = \|\bu_T(\bmm)\|^2 + \|\bm{u}_{\rm obs}\|^2 - 2\mathrm{Re}\braket{\bm{u}_{\rm obs},\bu_T(\bmm)}$} \\
\midrule
I   & 18.00 & $17.25\pm0.03$ & $5.02\times10^{-2}\pm4.05\times10^{-5}$ & $17.24\pm0.01$ & $5.02\times10^{-2}\pm1.39\times10^{-5}$ \\
II  & 14.00 & $13.84\pm0.13$ & $3.05\times10^{-2}\pm4.79\times10^{-3}$ & $13.93\pm0.03$ & $2.85\times10^{-2}\pm1.41\times10^{-4}$ \\
III & 10.00 & $9.84\pm0.04$  & $8.82\times10^{-3}\pm1.22\times10^{-4}$ & $9.83\pm0.04$  & $8.76\times10^{-3}\pm2.20\times10^{-4}$ \\

\midrule
\multicolumn{6}{c}{Normalized lifted loss $\mathcal{L}_{\rm norm}^{Y}(\bmm) = 2 - 2\frac{\mathrm{Re}\braket{\bY_{\rm obs},\bY_T(\bmm)}}{\|\bY_{\rm obs}\|\|\bY_T(\bmm)\|}$} \\
\midrule
I   & 18.00 & $17.26\pm0.09$ & $5.03\times10^{-2}\pm1.28\times10^{-4}$ & $17.24\pm0.01$ & $5.02\times10^{-2}\pm4.26\times10^{-6}$ \\
II  & 14.00 & $13.98\pm0.55$ & $4.04\times10^{-2}\pm2.17\times10^{-2}$ & $13.90\pm0.00$ & $2.87\times10^{-2}\pm7.31\times10^{-18}$ \\
III & 10.00 & $9.76\pm0.07$  & $9.68\times10^{-3}\pm1.34\times10^{-3}$ & $9.77\pm0.01$ & $9.06\times10^{-3}\pm1.92\times10^{-4}$ \\

\midrule
\multicolumn{6}{c}{Normalized loss $\mathcal{L}_{\rm norm}^{u}(\bmm) = 2 - 2\frac{\mathrm{Re}\braket{\bu_{\rm obs},\bu_T(\bmm)}}{\|\bu_{\rm obs}\|\|\bu_T(\bmm)\|}$} \\
\midrule
I   & 18.00 & $17.99\pm0.54$ & $5.96\times10^{-2}\pm1.45\times10^{-2}$ & $17.78\pm0.04$ & $5.43\times10^{-2}\pm6.45\times10^{-4}$ \\
II  & 14.00 & $14.01\pm0.02$ & $2.84\times10^{-2}\pm9.02\times10^{-5}$ & $14.02\pm0.01$ & $2.84\times10^{-2}\pm7.36\times10^{-5}$ \\
III & 10.00 & $7.88\pm3.43$  & $2.00\times10^{-1}\pm1.48\times10^{-1}$ & $10.98\pm0.09$ & $5.36\times10^{-2}\pm4.09\times10^{-3}$ \\
\bottomrule
\end{tabular}
}

\label{tab:Burger_results}
\end{table}

\begin{figure}
    \centering
    \includegraphics[width=1\linewidth]{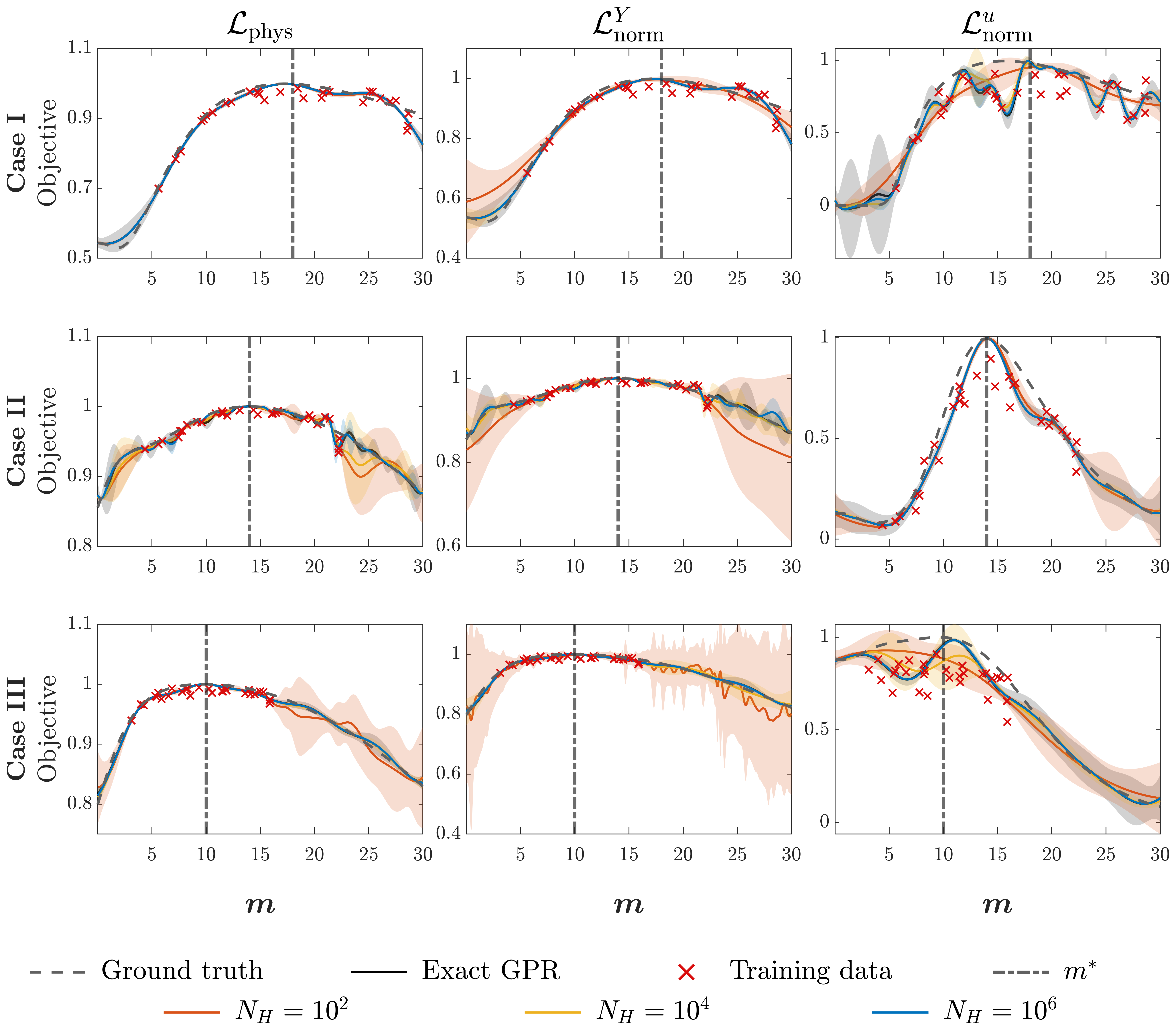}
    \caption{The inversion results for the forced viscous Burgers equation under different loss functions. The three columns correspond to the physical loss $\mathcal{L}_{\rm phys}(\bmm)$, the lifted normalized loss $\mathcal{L}_{\rm norm}^{Y}(\bmm)$, and the projected normalized loss $\mathcal{L}_{\rm norm}^{u}(\bmm)$, respectively.}
    \label{fig:Burger_results_inversion}
\end{figure}

The performance of BayesOpt is compared across three loss functions. We first neglect the Hadamard test sampling error and evaluate the normalized inner product exactly ($N_H = \infty$), and then report both the infinite-shot baseline and finite-shot Hadamard test emulations with $N_H\in\{10^2,10^4,10^6 \}$. \cref{tab:Burger_results} summarizes the predicted parameters and the corresponding inversion errors. Although the values of the objective function in the Carleman system become very close after normalization in the $D_N$-dimensional lifted space, BayesOpt is still able to exploit these small discrepancies and gradually converge to a neighborhood of the true parameter as the number of iterations increases. 

Overall, both quantum-state loss functions yield reasonable inversion results, demonstrating the feasibility of the proposed quantum loss-evaluation framework within the BayesOpt procedure. \cref{fig:Burger_results_inversion} shows the comparison between the trained GPR surrogate and the exact objective function under different loss functions and settings. The corresponding forward solutions and optimization trajectories are provided in the supplementary material. It is seen that the performance of the normalized lifted loss is comparable to that of the physical loss. The uncertainty associated with the normalized loss appears larger during inversion, which may facilitate broader exploration of the parameter space.

\section{Conclusion and discussion} \label{sec:conclusion}
In this work, we propose a quantum-assisted Bayesian optimization framework for PDE-based inverse problems. The quantum component evaluates the data-misfit loss by estimating quantum-state overlaps, while the classical component employs Bayesian optimization to select candidate parameters and trains a GPR surrogate model of the objective function. We show how to estimate loss functions between normalized vectors encoded in quantum states and how to reconstruct the corresponding physical loss with the correct classical norm scaling. We further analyze the computational complexity and error bounds of the proposed framework. Numerical experiments on a convection-diffusion equation using quantum-circuit simulations and a forced viscous Burgers equation using Carleman linearization are presented to demonstrate the feasibility and effectiveness of the proposed approach.

A limitation of the present quantum-assisted framework is the assumption of efficient oracle access to the observation state. For arbitrary dense and unstructured observational data, the associated state-preparation cost may be non-negligible. Nevertheless, observations in PDE-based Bayesian inverse problems are often collected from a sparse set of sensors \cite{alexanderian2014optimal,alexanderian2021optimal}, and additional sparsity or functional structures in the observational data may be exploited by specialized quantum state-preparation techniques \cite{zhang2022quantum,ramacciotti2024simple,mao2024toward,li2025nearly,alexanderian2021optimal}. Incorporating such structure-aware preparation strategies and explicitly accounting for their cost within the overall quantum-assisted inversion framework constitute important directions for future work. At the same time, the BayesOpt surrogate modeling  in our framework remains classical. 

It deserves to mention that quantum kernel methods \cite{schuld2021quantum} and quantum-assisted GPR \cite{zhao2019quantum} have recently been developed and may provide computational advantages for certain problem instances. Integrating the proposed quantum loss-evaluation procedure with quantum-enhanced surrogate modeling and optimization therefore represents a meaningful direction for future research.

\appendix

\section{Quantum PDE solvers}

We first consider an evolution PDE parameterized by $\bmm=(m_1,\ldots,m_d)$ in a feasible parameter set $\mathcal{M}$,
\begin{equation}
\label{eq:pde}
\frac{\partial u(\boldsymbol{x},t;\bmm)}{\partial t}  = \mathcal{F}\bigl(u(\boldsymbol{x},t;\bmm),t;\bmm\bigr),
\qquad u(\boldsymbol{x},0;\bmm)=u_0(\boldsymbol{x}),
\end{equation}
where $\mathcal{F}$ denotes a possibly linear or nonlinear differential operator and may also include source terms.

After spatial discretization, \cref{eq:pde} gives rise to a finite-dimensional dynamical system for the discrete solution vector $\bu(t;\bmm)$. We first consider the simplest case in which the resulting system is linear and autonomous,
\[
    \frac{d\bu(t;\bmm)}{dt} = A(\bmm)\bu(t;\bmm), \qquad \bu(0;\bmm)=\bu_0,
\]
where $A(\bmm)$ is the parameter-dependent, time-independent system matrix arising from the spatial discretization. The terminal solution has a closed form solution $\bu_T(\bmm) = e^{T A(\bmm)}\bu_0$. This matrix-exponential representation provides a natural starting point for block-encoding-based quantum algorithms that approximate the action of $e^{T A(\bmm)}$ through polynomial transformations of the encoded matrix, subject to the structural requirements of the corresponding polynomial-transformation framework \cite{low2017optimal,gutierrez2026quantum,motlagh2024generalized}.

To illustrate the construction of the loss functions in concrete settings, we consider two representative examples: a linear non-autonomous equation treated using the linear combination of Hamiltonian simulation (LCHS) technique \cite{an2023linear}, and a nonlinear equation treated by Carleman linearization within a generalized quantum signal processing (GQSP) framework \cite{gutierrez2026quantum,motlagh2024generalized}.

\subsection{Linear non-autonomous evolution via LCHS}

We next consider the semi-discretized linear non-autonomous system
\[
    \frac{\mathrm{d}\bu(t;\bmm)}{\mathrm{d}t}=A(t;\bmm)\bu(t;\bmm), \qquad \bu(0;\bmm)=\bu_0,
\]
where $\bu(t;\bmm)=(u(\bx_1,t;\bmm),\dots,u(\bx_{n_x},t;\bmm))^T$, and $A(t;\bmm)$ is generally time dependent and non-Hermitian. Let $\Phi_{\bmm}(t,s)$ denote the corresponding evolution operator. Then
\[
    \Phi_{\bmm}(t,s)=\mathcal{T}\exp\left(\int_s^t A(\tau;\bmm)\,\mathrm{d}\tau\right), \quad
   \bu_T(\bmm)=\Phi_{\bmm}(T,0)\bu_0,
\]
where $\mathcal{T}$ is the time ordering operator.

Following the LCHS framework \cite{an2023linear}, we assume that the Hermitian part of $A(t;\bmm)$ is negative semi-definite and decompose $A(t;\bmm)=-\bigl(L(t;\bmm)+\mathrm{i}H(t;\bmm)\bigr),$ where $L(t;\bmm)=-(A(t;\bmm)+A^\dagger(t;\bmm))/2\succeq0$, $H(t;\bmm)=(A^\dagger(t;\bmm)-A(t;\bmm))/(2\mathrm{i})$ are both Hermitian. Under this assumption, the nonunitary evolution operator admits the LCHS representation
\[
    \Phi_{\bmm}(T,0)=\int_{\mathbb{R}}\frac{1}{\pi(1+k^2)}U_k(T,0;\bmm)\,\mathrm{d}k .
\]
With
\[
    U_k(T,s;\bmm)=\mathcal{T}\exp\left[-\mathrm{i}\int_s^T\bigl(H(\tau;\bmm)+kL(\tau;\bmm)\bigr)\,\mathrm{d}\tau\right].
\]

Since $H(\tau;\bmm)+kL(\tau;\bmm)$ is Hermitian for every $k\in\mathbb{R}$, each $U_k(T,s;\bmm)$ is a unitary time evolution. Hence, LCHS represents the original nonunitary dynamics as a continuous linear combination of Hamiltonian simulations.

For a finite quantum implementation, the integral over $k$ is truncated to $[-K,K]$ and subsequently discretized. Let $\{k_j\}_{j=0}^{M}\subset[-K,K]$ and $\{\omega_j\}_{j=0}^{M}$ denote the quadrature nodes and weights, respectively. By defining $c_j=\omega_j/[\pi(1+k_j^2)]$, the terminal solution is approximated by
\[
    \widetilde{\bu}_T(\bmm)=\sum_{j=0}^{M}c_jU_{k_j}(T,0;\bmm)\bu_0.
\]

For the positive quadrature weights considered in the linear combinations of unitaries (LCU) implementation, define the normalization factor $\lambda(\bmm)=\sum_{j=0}^{M}c_j$. In the homogeneous setting considered here, $\lambda(\bmm)$ is determined by the truncation and quadrature rule; in particular, since the Cauchy kernel is normalized according to $\int_{\mathbb{R}}[\pi(1+k^2)]^{-1}\mathrm{d}k=1$, one has $\lambda(\bmm)\to1$ as the truncation errors vanish.

The discretized linear combination can then be implemented coherently through the standard LCU construction. In particular, the coefficient-preparation and selection unitaries satisfy
\[
    O_{\rm coef}\ket{0}^{\otimes a_k}=\frac{1}{\sqrt{\lambda(\bmm)}}\sum_{j=0}^{M}\sqrt{c_j}\ket{j}, \quad
    \operatorname{SELECT}(U)=\sum_{j=0}^{M}\ket{j}\!\bra{j}\otimes U_{k_j}(T,0;\bmm),
\]
where $a_k$ denotes the number of qubits in the coefficient register. After applying the controlled Hamiltonian simulations and reversing the coefficient-state preparation, the desired linear combination is encoded in the ancillary all-zero branch.

We assume that normalized amplitude encoding $\ket{\bu_0}$ can be prepared by an initial-state preparation oracle $O_{\rm prep}:\ket{0}\rightarrow\ket{\bu_0}$. Therefore, the LCHS circuit produces
\[
\begin{aligned}
    \ket{\psi_T(\bmm)} & := U_{\mathrm{solver}}(\bmm)\ket{0} = U_{\rm LCHS}(\bmm)\ket{0}^{\otimes a}\ket{\bu_0} \\
    & =\frac{\|\widetilde{\bu}_T(\bmm)\|}{\lambda(\bmm)\|\bu_0\|}\ket{0}^{\otimes a}\ket{\widetilde{\bu}_T(\bmm)} +\ket{\Phi_T^\perp(\bmm)},
\end{aligned}
\]
where $\ket{\widetilde{\bu}_T(\bmm)} = \sum_{j=0}^{{n_x}-1} \widetilde{u}_{T,j+1}(\bmm)\ket{j}/\|\widetilde{\bu}_T(\bmm)\|$ is the normalized amplitude encoding of $\widetilde{\bu}_T(\bmm)$, and $a$ denotes the total number of ancillary qubits used by the coherent LCHS implementation $\left(\bra{0}^{\otimes a}\otimes I\right)\ket{\Phi_T^\perp(\bmm)}=0.$

Thus, the LCHS construction naturally fits the general coherent-output form adopted in our framework: the approximate terminal solution is encoded in a designated successful branch, while all remaining ancillary and garbage components are orthogonal to that branch.

To apply the Hadamard test construction introduced in the main text, we define the normalized amplitude encoding of the observation $\bu_{\rm obs}=(u_{{\rm obs},1},\dots,u_{{\rm obs},{n_x}})^T$ as $\ket{\bu_{\rm obs}}.$ The corresponding observation state in the successful branch is 
\[
    \ket{\psi_{\rm obs}} = \ket{0}^{\otimes a} \otimes \left( \sum_{j=0}^{{n_x}-1} \frac{1}{\|\bu_{\rm obs}\|} u_{{\rm obs},j+1}\ket{j} \right). 
\]
We assume access to an observation-state oracle $O_{\rm obs}$ satisfying $\ket{\psi_{\rm obs}}=O_{\rm obs}\ket{0}$. To estimate the overlap, we define the unitary operator $W_u(\bmm)=O_{\rm obs}^\dagger U_{\rm solver}(\bmm)$. Given an auxiliary qubit initialized in $\ket{0}_H$, the Hadamard test circuit  reads \cite{nielsen2010quantum}
\begin{equation}  \label{eq:hadamard}
\begin{aligned}
    \ket{0}_H\ket{0}
    &\xrightarrow{H\otimes I}
    \frac{\ket{0}_H+\ket{1}_H}{\sqrt 2}\ket{0}
    \xrightarrow{c-W_u(\bmm)}
    \frac{1}{\sqrt 2}
    \big( \ket{0}_H\ket{0} +\ket{1}_H W_u(\bmm)\ket{0} \big) \\
    &\xrightarrow{H \otimes I}
    \frac{1}{2}\ket{0}_H
    \big( \ket{0} + W_u(\bmm)\ket{0} \big) +
    \frac{1}{2}\ket{1}_H
    \big( \ket{0} - W_u(\bmm)\ket{0} \big).
\end{aligned}
\end{equation}
Consequently, the probability of measuring the auxiliary qubit to be in state $\ket{0}_H$ is
\begin{equation*} 
    \begin{aligned}
    p^u_H(\bmm)
    &= \left\| \frac{1}{2}\left( \ket{0} + W_u(\bmm)\ket{0} \right)
    \right\|^2 \\
    &=  \frac{1}{4}
    \left( \braket{0|0} + \bra{0}W_u(\bmm)\ket{0} + \bra{0}W_u^\dagger(\bmm)\ket{0} + \bra{0}W_u^\dagger(\bmm)W_u(\bmm)\ket{0} \right) \\
    &= \frac{1}{2} \left( 1 + \mathrm{Re} \bra{0}W_u(\bmm)\ket{0}
    \right) = \frac{ 1+ \mathrm{Re} \braket{\psi_{\rm obs}|\psi_T(\bmm)} }{2}.
    \end{aligned}
\end{equation*}
Therefore, the inner product of the two quantum states can be estimated as
\[
    2p^u_H(\bmm)-1 
    = \mathrm{Re} \braket{\psi_{\rm obs}|\psi_T(\bmm)} 
    = \frac{ \mathrm{Re} \braket{ \bu_{\rm obs},
    \widetilde{\bu}_T(\bmm)}}{ \lambda(\bmm) \|\bu_0\| \|\bu_{\rm obs}\|}.
\]
The LCHS success probability $p_{\rm succ}(\bmm)$ is given by
\[
    p_{\rm succ}(\bmm) 
    = \left\| \left(\bra{0}^{\otimes a}\otimes I\right)U_{\mathrm{solver}}(\bmm)\ket{0} \right\|^2
    = \frac{ \left\|\widetilde{\bu}_T(\bmm)\right\|^2}{ \lambda(\bmm)^2\left\|\bu_0\right\|^2 }.
\]

Taking their ratio eliminates the LCHS normalization factor and the norm of the initial state,
\[
    \frac{2p_H^u(\bmm)-1}{\sqrt{p_{\rm succ}(\bmm)}}
    =\frac{\operatorname{Re}\langle\bu_{\rm obs},\widetilde{\bu}_T(\bmm)\rangle}{\|\bu_{\rm obs}\|\|\widetilde{\bu}_T(\bmm)\|}.
\]

Consequently, the infinite-shot normalized loss retains exactly the same form as in the general framework,
\[
    \widetilde{\mathcal{L}}_{\rm norm}^{u}(\bmm) = 2 - 2\frac{\operatorname{Re}\langle\bu_{\rm obs},\widetilde{\bu}_T(\bmm)\rangle}{\|\bu_{\rm obs}\|\|\widetilde{\bu}_T(\bmm)\|}  =2-\frac{4p_H^u(\bmm)-2}{\sqrt{p_{\rm succ}(\bmm)}},
\]
whereas the physical loss is reconstructed by retaining the LCU normalization,
\[
    \widetilde{\mathcal{L}}_{\rm phys}(\bmm)
    =\lambda^2(\bmm)\|\bu_0\|^2p_{\rm succ}(\bmm)+\|\bu_{\rm obs}\|^2
    -2\lambda(\bmm)\|\bu_0\|\|\bu_{\rm obs}\|\bigl(2p_H^u(\bmm)-1\bigr).
\]
This example shows explicitly that the loss-evaluation procedure does not depend on the internal Hamiltonian-simulation structure of LCHS once its coherent terminal-state output has been prepared.

The same construction can be extended to the inhomogeneous non-autonomous system
\[
    \frac{\mathrm{d}\bu(t;\bmm)}{\mathrm{d}t}=A(t;\bmm)\bu(t;\bmm)+\boldsymbol{b}(t;\bmm).
\]

In this case, the variation-of-constants formula gives
\[
    \bu_T(\bmm)=\Phi_{\bmm}(T,0)\bu_0+\int_0^T\Phi_{\bmm}(T,s)\boldsymbol{b}(s;\bmm)\,\mathrm{d}s.
\]
Applying the LCHS representation to the propagator $\Phi_{\bmm}(T,s)$ yields
\[
    \bu_T(\bmm)
    =\int_{\mathbb{R}}\frac{U_k(T,0;\bmm)\bu_0}{\pi(1+k^2)}\,\mathrm{d}k
    +\int_0^T\int_{\mathbb{R}}\frac{U_k(T,s;\bmm)\boldsymbol{b}(s;\bmm)}{\pi(1+k^2)}\,\mathrm{d}k\,\mathrm{d}s.
\]

Following \cite{an2023linear}, the second term is discretized with respect to both the LCHS variable $k$ and the time variable $s$, and the resulting source contributions are coherently combined with the initial-value contribution through an additional LCU construction, under the corresponding state-preparation and oracle assumptions for $\boldsymbol{b}(s;\bmm)$. The complete circuit therefore again prepares an approximation to $\bu_T(\bmm)$ in a designated successful branch.

For this inhomogeneous construction, we retain the notation of the general coherent output form and define the effective factor $\lambda(\bmm)$ such that $\lambda(\bmm)\|\bu_0\|$ is the overall sub-normalization factor of the complete LCU procedure, incorporating both the LCHS discretization and the coherent combination of the initial-value and source contributions. In contrast to the homogeneous case, this effective $\lambda(\bmm)$ is therefore generally not given solely by the sum of the LCHS quadrature coefficients. Once the terminal-state output is written in this general coherent form, no modification of the subsequent loss-evaluation procedure is required.

\subsection{Carleman linearization via GQSP}

For the nonlinear case, Carleman linearization embeds a finite-dimensional polynomial nonlinear system into an infinite-dimensional linear system by introducing tensor powers of the original state \cite{liu2021efficient,wu2025quantum}. We focus on the setting in which, after spatial discretization, the nonlinear evolution takes the autonomous quadratic form
\[
    \frac{\mathrm{d}\bu(t;\bmm)}{\mathrm{d}t}
    =F_2(\bmm)\bu^{\otimes2}(t;\bmm)+F_1(\bmm)\bu(t;\bmm)+F_0(\bmm),
    \quad \bu(0;\bmm)=\bu_0,
\]
where $\bu^{\otimes2}=\bu\otimes\bu\in\mathbb{R}^{n_x^2}$, $F_1(\bmm)\in\mathbb{R}^{n_x\times n_x}$, $F_2(\bmm)\in\mathbb{R}^{n_x\times n_x^2}$, and $F_0(\bmm)\in\mathbb{R}^{n_x}$ denote the linear, quadratic, and time-independent forcing components, respectively.

The Carleman hierarchy is generated by the tensor powers of the original state, together with the constant component $1$. For a prescribed truncation order $N$, we retain the tensor levels from $0$ to $N$ and neglect the coupling to tensor powers of order higher than $N$. We denote the resulting truncated Carleman state by $\bY_N(t;\bmm)\in\mathbb{R}^{D_N}, D_N=\sum_{j=0}^{N}n_x^j.$

The constant component incorporates the affine forcing $F_0(\bmm)$ into the lifted linear dynamics. The resulting finite-dimensional truncated system is
\begin{equation*}
    \frac{d\bY_N(t;\bmm)}{dt}  = A(\bmm)\bY_N(t;\bmm), \qquad \bY_N(0;\bmm)=\bY_0,
\end{equation*}
where $\bY_0  =\left( 1, \bu_0, \bu_0^{\otimes2}, \ldots, \bu_0^{\otimes N}\right)^T.$ The truncated matrix $A(\bmm)\in\mathbb{R}^{D_N\times D_N}$ has the block-tridiagonal structure
\[
A(\bmm)=
\begin{pmatrix}
0 & 0 & 0 & \cdots & 0\\
A_1^0(\bmm) & A_1^1(\bmm) & A_1^2(\bmm) & \cdots & 0\\
0 & A_2^1(\bmm) & A_2^2(\bmm) & A_2^3(\bmm) & \cdots\\
\vdots & \ddots & \ddots & \ddots & \vdots\\
0 & \cdots & 0 & A_N^{N-1}(\bmm) & A_N^N(\bmm)
\end{pmatrix}.
\]
The block $A^j_i(\bmm)$ represents the coupling from the $j$-th tensor level to the $i$-th tensor level, where the zeroth tensor is defined as $\bu^{\otimes0}=1$. Its nonzero blocks are given by
\begin{align*}
A_j^{j-1}(\bmm)
&=\sum_{\ell=0}^{j-1}I_{n_x}^{\otimes\ell}\otimes F_0(\bmm)\otimes I_{n_x}^{\otimes(j-1-\ell)},
&&1\leq j\leq N,\\
A_j^{j}(\bmm)
&=\sum_{\ell=0}^{j-1}I_{n_x}^{\otimes\ell}\otimes F_1(\bmm)\otimes I_{n_x}^{\otimes(j-1-\ell)},
&&1\leq j\leq N,\\
A_j^{j+1}(\bmm)
&=\sum_{\ell=0}^{j-1}I_{n_x}^{\otimes\ell}\otimes F_2(\bmm)\otimes I_{n_x}^{\otimes(j-1-\ell)},
&&1\leq j\leq N-1.
\end{align*}
Here, $I_{n_x}$ denotes the ${n_x}\times {n_x}$ identity matrix and $I_{n_x}^{\otimes0}=1$. The first block row vanishes since $\mathrm{d}\by_0/\mathrm{d}t=0$, while the first-order component of the truncated lifted solution provides an approximation to the original nonlinear state.

The truncated nonlinear evolution is therefore reduced to a finite-dimensional linear autonomous system, with terminal lifted solution
\[
    \bY_N(T;\bmm)=e^{TA(\bmm)}\bY_0.
\]

To approximate this terminal evolution, let $\alpha(\bmm)>0$ satisfy $\|A(\bmm)/\alpha(\bmm)\|\leq1$ and define the degree-$n$ normalized Taylor polynomial
\begin{equation}\label{eq:poly_approx}
    P_n\left(\frac{A(\bmm)}{\alpha(\bmm)}\right)
    =\frac{1}{\lambda_{Y}(\bmm)}\sum_{k=0}^{n}\frac{(T\alpha(\bmm))^k}{k!}
    \left(\frac{A(\bmm)}{\alpha(\bmm)}\right)^k
    \approx\frac{1}{\lambda_{Y}(\bmm)}e^{TA(\bmm)},
\end{equation}
where $\lambda_{Y}(\bmm)=\sum_{k=0}^{n}(T\alpha(\bmm))^k/k!$. The corresponding scalar polynomial
satisfies $|P_n(z)|\leq1$ for $|z|\leq1$, making it compatible with GQSP. To implement $P_n(\cdot)$ for the generally non-Hermitian Carleman matrix, we employ the $n$-regular block-encoding construction of \cite{gutierrez2026quantum}.

\subsubsection{Quantum algorithm for the forward solver}
We first introduce the definition of an $n$-regular block-encoding.

\begin{definition}[$n$-regular block-encoding \cite{gutierrez2026quantum}]
An $(a,\epsilon_{\rm BE})$-block-encoding $U_A$ of a normalized square matrix $A$ is called an $n$-regular block-encoding if, for every $0\leq k\leq n$, $U_A^k$ is an $(a,k\epsilon_{\rm BE})$-block-encoding of $A^k$, i.e.,
\[
    \left\|(\bra{0}^{\otimes a}\otimes I)U_A^k
    (\ket{0}^{\otimes a}\otimes I)-A^k\right\|
    \leq k\epsilon_{\rm BE}.
\]
Here, $A$ satisfies $\|A\|\leq1$.
\end{definition}

This regularity prevents unsuccessful branches generated by one application of the block-encoding from returning to the successful subspace under subsequent applications. Following \cite{gutierrez2026quantum}, it can be achieved by introducing an additional counter register. For a degree $n=2^b$, let $Q_n$ denote the $b$-qubit incrementer
\begin{equation}\label{eq:incrementer}
    Q_n\ket{x}=\ket{(x+1)\bmod n},
    \qquad x\in\{0,1,\ldots,n-1\}.
\end{equation}
For a general degree $n$, the logical $n$-cycle can be embedded into
$b=\lceil\log_2n\rceil$ qubits and extended unitarily to the orthogonal complement.
The incrementer is applied whenever the ancillary register of the original
block-encoding is not in its all-zero state, thereby preventing failed branches
from returning to the successful subspace during subsequent applications.

Based on this regularized block-encoding, we employ GQSP \cite{motlagh2024generalized} to implement the polynomial transformation defined in \cref{eq:poly_approx}.

\begin{definition}[Generalized quantum signal processing \cite{motlagh2024generalized}]
Let $P(z)$ be a polynomial of degree at most $n$ satisfying $|P(z)|\leq1$ for all
$z\in\mathbb{T}:=\{z\in\mathbb{C}:|z|=1\}$. Then there exists a sequence of
single-qubit operations $R_0,\ldots,R_n\in SU(2)$ such that
\[
    \bra{0}R_0\widetilde{w}R_1\widetilde{w}\cdots
    \widetilde{w}R_n\ket{0}=P(z),
\]
where $\widetilde{w}=\operatorname{diag}(1,z)$ denotes the scalar signal operator.
\end{definition}

The single-qubit operations $R_0,\ldots,R_n$ can be parameterized by Pauli rotations, with their rotation angles computed using numerical synthesis algorithms \cite{alexis2026infinite,dong2024robust,ying2022stable}. Replacing the scalar signal $z$ by a unitary operator gives the matrix-valued signal $\widetilde{W}(U)=\operatorname{diag}(I,U)$.

The query complexity for an arbitrary square matrix is summarized as follows.

\begin{theorem}[Quantum eigenvalue transformation \cite{gutierrez2026quantum}]
\label{the:QET}
Let $A$ be an arbitrary square matrix satisfying $\|A\|\leq 1$, and let $P(z)$ be a polynomial of degree $n$ satisfying $|P(z)|\leq 1$ for all $z\in\mathbb D:=\{z\in\mathbb C:|z|\le 1\}$. Given an $n$-regular $(a,\epsilon_{\rm BE})$-block-encoding of $A$, the construction produces a block-encoding of $P(A)$ with block-encoding error bounded by $\sqrt{{n^3}/{3}}\,\epsilon_{\rm BE}.$ The circuit requires $\mathcal{O}(\log n)$ additional ancillary qubits, $\mathcal{O}(n\log n)$ elementary operations, and $n$ calls to the block-encoding of $A$.
\end{theorem}

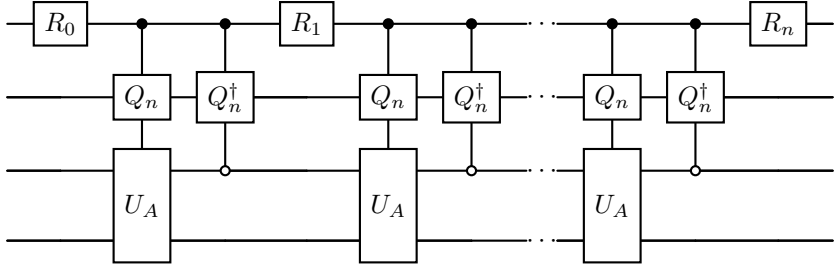
\begin{figure}[ht]
\centering
\begin{quantikz}[row sep=0.35cm, column sep=0.35cm]
& \gate{R_0}& \ctrl{2}& \ctrl{1}& \gate{R_1}& \ctrl{2}& \ctrl{1}&\cdots& \ctrl{2}& \ctrl{1}& \gate{R_n}& \qw \\
& \qw& \gate{Q_n}& \gate{Q_n^\dagger}& \qw& \gate{Q_n}& \gate{Q_n^\dagger}& \cdots& \gate{Q_n}& \gate{Q_n^\dagger}& \qw& \qw \\
& \qw& \gate[2]{U_A}& \octrl{-1}& \qw& \gate[2]{U_A}& \octrl{-1}& \cdots& \gate[2]{U_A}& \octrl{-1}& \qw& \qw \\
& \qw& \qw& \qw& \qw& \qw& \qw& \cdots& \qw& \qw& \qw& \qw
\end{quantikz}
\caption{Circuit for polynomial transformation of a general matrix using an $n$-regular block-encoding and GQSP.}
\label{fig:GQSP}
\end{figure}

Unlike quantum singular value transformation (QSVT) \cite{gilyen2019quantum}, which in general transforms singular values of a block-encoded matrix, the regularized GQSP construction directly implements the matrix polynomial $P(A)$ for an arbitrary square matrix. The complete construction is illustrated in \cref{fig:GQSP}.

We now apply this construction to the normalized Carleman matrix $A(\bmm)/\alpha(\bmm)$. Let the system register contain $s=\lceil\log_2D_N\rceil$ qubits, and suppose that an $(a,\epsilon_{\rm BE})$-block-encoding of $A(\bmm)/\alpha(\bmm)$ is available. Introducing a $b=\lceil\log_2n\rceil$-qubit counter register yields an $n$-regular block-encoding, which we continue to denote by $U_A(\bmm)$.

With one additional GQSP processing qubit, define the signal unitary
\[
    \widetilde{W}(U_A) =\operatorname{diag}\left(I^{\otimes(b+a+s)},U_A(\bmm)\right).
\]
For the given polynomial $P_n(A(\bmm)/\alpha(\bmm))$, the processing operations $R_0,\ldots,R_n\in SU(2)$ can be computed classically using numerical synthesis algorithms \cite{alexis2026infinite,dong2024robust,ying2022stable}. Therefore, the resulting GQSP unitary reads
\[
    U_{\rm GQSP}(\bmm) =\bigl(R_0\otimes I\bigr)\widetilde{W}(U_A) \bigl(R_1\otimes I\bigr)\widetilde{W}(U_A)\cdots \widetilde{W}(U_A)\bigl(R_n\otimes I\bigr).
\]

Define the initial-state preparation oracle $O_{\rm prep}^Y:\ket{0}\rightarrow \ket{\bY_0}$ for the Carleman linearization system, where $\ket{\bY_0} =\sum_{\ell=0}^{D_N-1}\bY_{0,\ell}\ket{\ell} /\|\bY_0\|$ is the normalized amplitude encoding of the initial lifted vector $\bY_0$. Let $\ket{\psi_0} =\ket{0}\ket{0}^{\otimes b}\ket{0}^{\otimes a} \otimes\ket{\bY_0}$ denote the input state of the full circuit. Applying the GQSP circuit under the ideal block-encoding assumption gives
\begin{equation}\label{eq:car_state}
\begin{aligned}
    \ket{\psi_T^Y(\bmm)} &:= U^Y_{\rm solver}\ket{0} =U_{\rm GQSP}(\bmm)\ket{\psi_0} \\
    &=\ket{0}\ket{0}^{\otimes b}\ket{0}^{\otimes a} \otimes P_n\left(\frac{A(\bmm)}{\alpha(\bmm)}\right)\ket{\bY_0} +\ket{\Phi_T^{Y,\perp}(\bmm)}.
\end{aligned}
\end{equation}
Here, $\ket{\Phi_T^{Y,\perp}(\bmm)}$ is orthogonal to the successful subspace in which the processing qubit, counter register, and block-encoding ancillas are all in the zero state.

To connect this output with the general coherent-output representation introduced above, define the approximated vector $\widetilde{\bY}_{N,T}(\bmm):=\lambda_{Y}(\bmm)\, P_n\left(\frac{A(\bmm)}{\alpha(\bmm)}\right)\bY_0.$ Then \cref{eq:car_state} can equivalently be written as
\[
    \ket{\psi_T^Y(\bmm)}
    =\frac{\|\widetilde{\bY}_{N,T}(\bmm)\|}
    {\lambda_{Y}(\bmm)\|\bY_0\|}
    \ket{0}\ket{0}^{\otimes b}\ket{0}^{\otimes a}
    \ket{\widetilde{\bY}_{N,T}(\bmm)}
    +\ket{\Phi_T^{Y,\perp}(\bmm)},
\]
where $\ket{\widetilde{\bY}_{N,T}(\bmm)}$ denotes the normalized amplitude encoding of $\widetilde{\bY}_{N,T}(\bmm)$. Thus, for this GQSP realization, the normalization factor in the general lifted-state output is given by $\lambda_{Y}(\bmm)$. The corresponding successful lifted-state probability is
\[
    p_{\rm succ}^{Y}(\bmm)
    =\left\|P_n\left(\frac{A(\bmm)}{\alpha(\bmm)}\right)\ket{\bY_0}\right\|^2
    =\frac{\|\widetilde{\bY}_{N,T}(\bmm)\|^2}
    {\lambda_{Y}(\bmm)^2\|\bY_0\|^2}.
\]
This coherent output can therefore be used directly in the normalized lifted loss-evaluation procedures with an additional Hadamard test circuit.

\subsubsection{Loss reconstruction from the GQSP output}

We next specialize the three loss functions to the GQSP construction above. For the normalized loss in the full Carleman space, the lifted observation vector is defined as 
\[
\bY_{\rm obs} =\left(1,\bu_{\rm obs},\bu_{\rm obs}^{\otimes2},\dots,\bu_{\rm obs}^{\otimes N}\right)^T \in\mathbb{R}^{D_N}.
\]
 The corresponding normalized observation state for $\bY_{\rm obs}$ in the full computational space is encoded by
\[
    \ket{\psi_{\rm obs}^Y} :=
    \ket{0}\ket{0}^{\otimes b}\ket{0}^{\otimes a}\otimes \left( \frac{1}{\|\bY_{\rm obs}\|} \sum_{\ell=0}^{D_N-1}Y_{{\rm obs},\ell+1}\ket{\ell} \right) =O_{\rm obs}^Y\ket{0}.
\]

Define the unitary operator $W_Y(\bmm)=(O_{\rm obs}^Y)^\dagger U_{\rm solver}^Y(\bmm)$. The Hadamard test follows the same procedure as in \cref{eq:hadamard}. The target real part of the inner product is then given by
\[
    2p_H^Y(\bmm)-1 =\mathrm{Re}\braket{\psi_{\rm obs}^Y|\psi_T^Y(\bmm)}
    =\frac{\mathrm{Re}\langle\bY_{\rm obs},\widetilde{\bY}_{N,T}(\bmm)\rangle}{\lambda_Y(\bmm)\|\bY_{0}\|\|\bY_{\rm obs}\|}.
\]
The corresponding GQSP success probability is
\[
    p_{\rm succ}^{Y}(\bmm)
    =\left\|P_n\left(\frac{A(\bmm)}{\alpha(\bmm)}\right)\ket{\bY_0}\right\|^2
    =\frac{\|\widetilde{\bY}_{N,T}(\bmm)\|^2}{\lambda_{Y}(\bmm)^2\|\bY_0\|^2}.
\]
After eliminating the normalization factor by combining these two quantities, the infinite-shot normalized loss in the lifted Carleman space is
\[
    \widetilde{\mathcal{L}}_{\rm norm}^Y(\bmm) = 2 - 2\frac{\operatorname{Re}\langle\bY_{\rm obs},\widetilde{\bY}_{N,T}(\bmm)\rangle}{\|\bY_{\rm obs}\|\|\widetilde{\bY}_{N,T}(\bmm)\|}
    =2-\frac{4p_H^Y(\bmm)-2}{\sqrt{p_{\rm succ}^{Y}(\bmm)}}.
\]

The same GQSP output can be used to evaluate losses defined on the original physical variable. At the classical level, let $\Pi_1:\mathbb{R}^{D_N}\rightarrow\mathbb{R}^{n_x}$
denote the extraction operator for the first-order Carleman component. Then $\bu_T(\bmm)=\Pi_1\bY_N(T;\bmm)$ gives the order-$N$ Carleman approximation to the exact terminal solution $\bu_T(\bmm)$. The terminal physical vector associated with the polynomial approximation is defined by $\widetilde{\bu}_T(\bmm) = \Pi_1\widetilde{\bY}_{N,T}(\bmm).$ In the quantum implementation, this physical component is identified coherently using a flag register $F$. Since the Carleman basis is ordered by tensor degree, with the constant component followed by the first-order component, the physical block corresponds to the system indices $1,\ldots,n_x$. We therefore introduce the reversible index-checking operation
\begin{equation*}
    O_u\ket{\ell}\ket{0}_F=
\begin{cases}
    \ket{\ell}\ket{1}_F, & \ell=0,\ldots,{n_x}-1,\\
    \ket{\ell}\ket{0}_F, & \mathrm{otherwise}.
\end{cases}
\end{equation*}
Applying $O_u$ to the GQSP output with the flag initialized in $\ket{0}_F$ gives the flagged terminal state $\ket{\psi_T^u(\bmm)} =O_u\left(\ket{\psi_T^Y(\bmm)}\otimes\ket{0}_F\right).$

To evaluate the physical-space overlap, we use the embedded observation vector $\bY_{\rm obs}^u =\left(0,\bu_{\rm obs},0,\dots,0\right)^T\in\mathbb{R}^{D_N}$. The corresponding observation state is prepared with the flag register in $\ket{1}_F$ and all GQSP ancillary registers in their successful zero states. Consequently, only the flagged physical component of the GQSP output contributes to the Hadamard test overlap, giving
\[
    2p_H^u(\bmm)-1 =\frac{\operatorname{Re}\langle\bu_{\rm obs},\widetilde{\bu}_T(\bmm)\rangle}
    {\lambda_{Y}(\bmm)\|\bY_0\|\|\bu_{\rm obs}\|}.
\]
The corresponding projected success probability is obtained by measuring the flag register in $\ket{1}_F$ while the GQSP ancillary registers are simultaneously measured in their successful all-zero state. It satisfies
\[
    p_{\rm succ}(\bmm)
    =\left\|\Pi_1P_n\left(\frac{A(\bmm)}{\alpha(\bmm)}\right)\ket{\bY_0}\right\|^2
    =\frac{\|\widetilde{\bu}_T(\bmm)\|^2}{\lambda_{Y}(\bmm)^2\|\bY_0\|^2}.
\]
Here, $\Pi_1$ describes the corresponding first-order component at the classical vector level, whereas the quantum implementation of this selection is realized by $O_u$ and the subsequent measurement of the flag register.

These two measurements lead directly to the normalized loss function
\[
    \widetilde{\mathcal{L}}_{\rm norm}^u(\bmm)
    =2-2\frac{\operatorname{Re}\langle\bu_{\rm obs},\widetilde{\bu}_T(\bmm)\rangle}
    {\|\bu_{\rm obs}\|\|\widetilde{\bu}_T(\bmm)\|}
    =2-\frac{4p_H^u(\bmm)-2}{\sqrt{p_{\rm succ}(\bmm)}}.
\]

Finally, the same quantities $p_H^u(\bmm)$ and $p_{\rm succ}(\bmm)$ can be used to reconstruct the physical loss. In particular, the physical norm can therefore be reconstructed as
\[
    \|\widetilde{\bu}_T(\bmm)\|^2 = \lambda_{Y}(\bmm)^2\|\bY_0\|^2p_{\rm succ}(\bmm),
\]
and the real part of the inner product is obtained from the Hadamard test statistics:
\[
    \operatorname{Re}\langle\bu_{\rm obs},\widetilde{\bu}_T(\bmm)\rangle
    =\lambda_{Y}(\bmm)\|\bY_0\|\|\bu_{\rm obs}\|\bigl(2p_H^u(\bmm)-1\bigr).
\]
Substituting these relations into $\widetilde{\mathcal{L}}_{\rm phys}(\bmm) =\|\widetilde{\bu}_T(\bmm)-\bu_{\rm obs}\|^2$ gives
\[
    \widetilde{\mathcal{L}}_{\rm phys}(\bmm)
    =\lambda_{Y}(\bmm)^2\|\bY_0\|^2p_{\rm succ}(\bmm)+\|\bu_{\rm obs}\|^2
    -2\lambda_{Y}(\bmm)\|\bY_0\|\|\bu_{\rm obs}\|
    \bigl(2p_H^u(\bmm)-1\bigr).
\]

Thus, the lifted normalized loss is evaluated directly from the successful GQSP branch, while the normalized loss and the physical loss additionally use one flag qubit to identify the first-order Carleman component. The former depends on $p_H^Y(\bmm)$ and $p_{\rm succ}^{Y}(\bmm)$, whereas the latter two are reconstructed from $p_H^u(\bmm)$ and $p_{\rm succ}(\bmm)$.

\section{Detailed results for the 1-D forced viscous Burgers equation}
In this section, we provide detailed Bayesian optimization results for the 1-D forced viscous Burgers equation to complement the inversion results presented in the main text. Specifically, we report the corresponding forward solutions, optimization trajectories, and the objective functions obtained with different loss functions. The results are based on the infinite-shot baseline to characterize the performance of the proposed framework without statistical sampling errors. 

We briefly review the problem setting as follows. The governing equation is
\[
    \partial_t u + u\partial_x u = \nu \partial^2_x u  + f, \quad (x,t)\in[-L_0/2,L_0/2]\times[0,T].
\]
The initial condition $u(x,0)=U_0\sin(2\pi x/L_0)$ is imposed on $[-L_0/2,L_0/2]$ with $L_0=1$, together with homogeneous Dirichlet boundary conditions $u(-L_0/2,t)=u(L_0/2,t)=0$. We choose the forcing term $f(x)=U_0\exp\left(-\frac{(x-L_0/4)^2}{2(L_0/32)^2}\right)$, where $U_0=\frac{1}{\sqrt{{n_x}-1}}$ and $n_x=16$. The viscosity is parameterized as $\nu=U_0L_0/\mathrm{Re}$, where $\mathrm{Re}$ denotes the Reynolds number. Three cases with different Reynolds numbers and final times are considered in this experiment. 

We benchmark the performance of Bayesian optimization (BayesOpt) using different loss functions. At this stage, we neglect Hadamard-test sampling error and present the detailed BayesOpt trajectories obtained with corresponding loss functions. The three loss functions can be written as follows:
\begin{enumerate}

\item[(1)] Physical loss $\mathcal{L}_{\textup{phys}}(\bmm)$: 
\[
\mathcal{L}_{\textup{phys}}(\bmm) = \|\bu_T(\bmm)\|^2 + \|\bm{u}_{\rm obs}\|^2 - 2 \mathrm{Re}\braket{\bm{u}_{\rm obs},\bu_T(\bmm)}.
\]
\item[(2)] Normalized lifted loss $\mathcal{L}_{\rm norm}^{Y}(\bmm)$:
\[
    \mathcal{L}^{Y}_{\mathrm{norm}}(\bmm) = \left\|\frac{\bY_T(\bmm)}{\|{\bY_T(\bmm)}\|} -\frac{\bY_{\mathrm{obs}}}{\|{\bY_{\mathrm{obs}}}\|}\right\|^2 = 2 - 2 \frac{\mathrm{Re}\braket{\bY_T(\bmm),\bY_{\mathrm{obs}}}}  {\|{\bY_T(\bmm)}\|\|{\bY_{\mathrm{obs}}}\|}.
\]
\item[(3)] Normalized loss $\mathcal{L}_{\rm norm}^{u}(\bmm)$:
\[
    \mathcal{L}^{u}_{\mathrm{norm}}(\bmm) = \left\|\frac{\bu_T(\bmm)}{\|\bu_T(\bmm)\|}-
    \frac{\bu_{\mathrm{obs}}}{\|\bu_{\mathrm{obs}}\|}\right\|^2
    = 2-2 \frac{\mathrm{Re}\braket{\bu_{\mathrm{obs}},\bu_T(\bmm)}}{\|\bu_T(\bmm)\|\|\bu_{\mathrm{obs}}\|}.
\]
\end{enumerate}
The objective function $f(\bmm)$ is defined based on the corresponding loss function: 
\[
    f(\bmm) = \exp\left( -\frac{\mathcal{L}(\bmm)}{n_{\rm obs}\cdot\gamma}  \right),
\]
where $\gamma=1/100$ is a scaling factor, and $n_{\rm obs}$ denotes the dimension of the corresponding observation vector, which depends on the selected loss function. Specifically, $n_{\rm obs}=n_x=16$ for $\mathcal{L}_{\mathrm{phys}}(\bmm)$ and $\mathcal{L}_{\mathrm{norm}}^{u}(\bmm)$, while $n_{\rm obs}=\sum_{j=0}^{N}n_x^j$ for $\mathcal{L}_{\mathrm{norm}}^{Y}(\bmm)$ with an $N$-order Carleman truncation.

\begin{figure}[ht]
    \centering
    \subfloat[Case I with $\mathrm{Re}=18, T=1$.]
    {\includegraphics[width=\linewidth,height=0.3\linewidth]{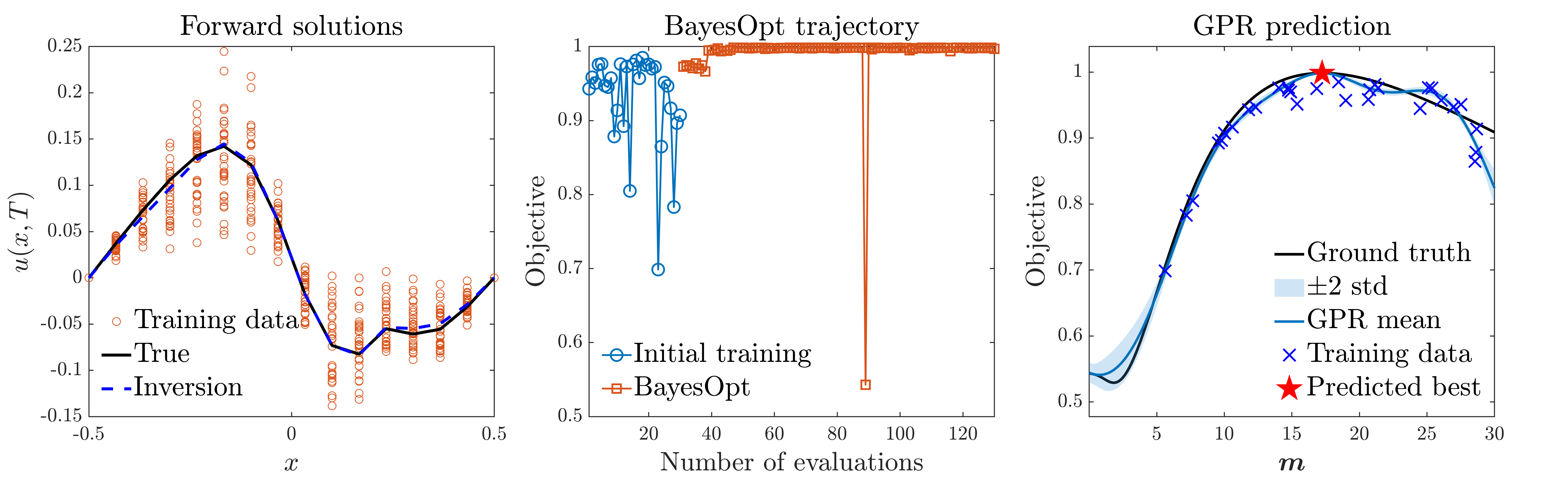}}\\
    \subfloat[Case II with $\mathrm{Re}=14, T=2$.]
    {\includegraphics[width=\linewidth,height=0.3\linewidth]{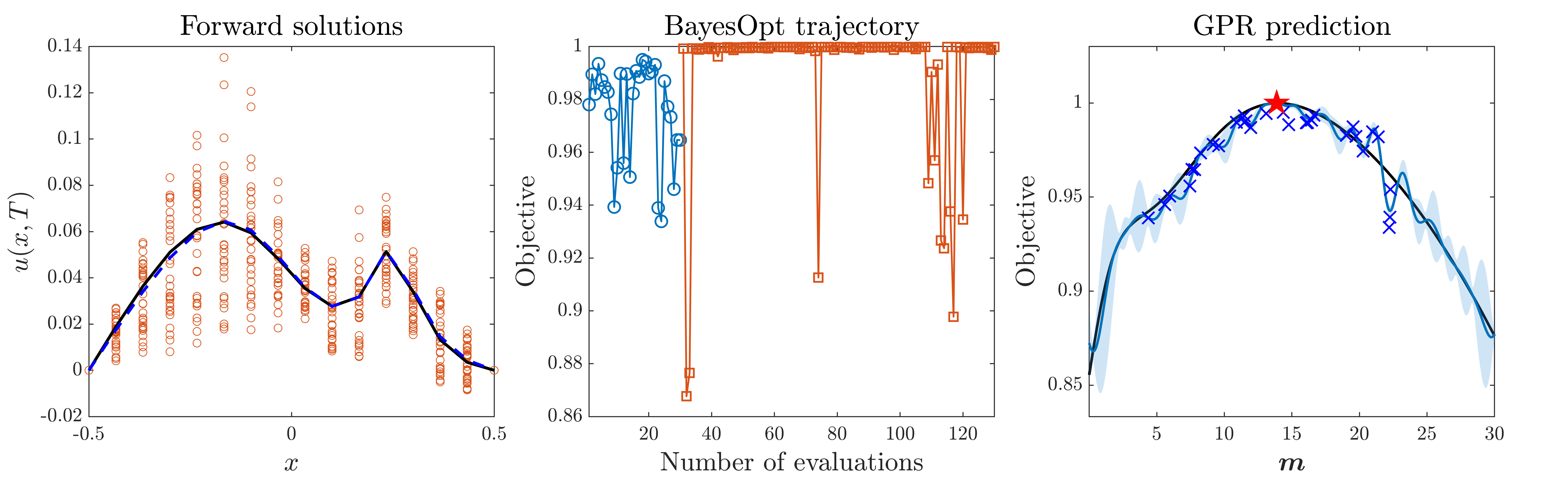}}\\
    \subfloat[Case III with $\mathrm{Re}=10, T=3$.]
    {\includegraphics[width=\linewidth,height=0.3\linewidth]{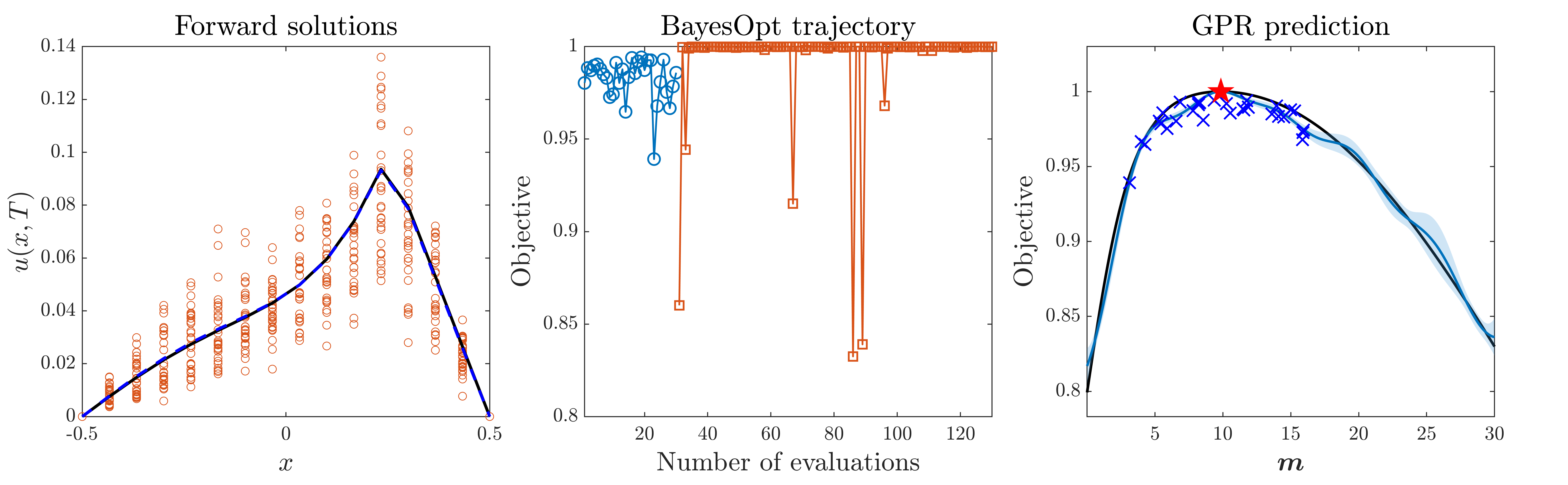}}\\
    \caption{The inversion results for the Burgers equation using the physical loss $\mathcal{L}_{\rm phys}(\bmm)$. }
   \label{fig:Burger_phys_loss}
\end{figure}

\begin{table}[ht]
    \centering
    \footnotesize
    \caption{Parameter-inversion results for the forced viscous Burgers equation: GPR-predicted parameters $\bmm_{\rm opt} = \arg\max_{\bmm} \mathbb{E}[f(\bmm) \mid \mathcal{D}]$ and the corresponding relative $L^2$ errors.}
    \resizebox{\linewidth}{!}{
    \begin{tabular}{cccccccc}
        \toprule
        \multirow{2}{*}{Case} & \multirow{2}{*}{True $\bmm^*$} & \multicolumn{3}{c}{Estimated $\bmm_{\rm opt}$} & \multicolumn{3}{c}{Relative $L^2$ error} \\
        \cmidrule(lr){3-5}\cmidrule(lr){6-8}
        & & $\mathcal{L}_{\textup{phys}}(\bmm)$ & $\mathcal{L}_{\rm norm}^{Y}(\bmm)$ & $\mathcal{L}_{\rm norm}^{u}(\bmm)$ & $\mathcal{L}_{\textup{phys}}(\bmm)$ & $\mathcal{L}_{\rm norm}^{Y}(\bmm)$ & $\mathcal{L}_{\rm norm}^{u}(\bmm)$ \\
        \midrule
        I   & 18 & 17.24 & 17.24 & 17.81 & $5.02\times10^{-2}$ & $5.02\times10^{-2}$ & $5.48\times10^{-2}$ \\
        II  & 14 & 13.87 & 13.90 & 14.02 & $2.91\times10^{-2}$ & $2.87\times10^{-2}$ & $2.85\times10^{-2}$ \\
        III & 10 & 9.85  & 9.76  & 10.90 & $8.66\times10^{-3}$ & $9.18\times10^{-3}$ & $5.02\times10^{-2}$ \\
        \bottomrule
    \end{tabular}}
    \label{tab:Burger_results_2}
\end{table}

\begin{figure}[H]
    \centering
    \subfloat[Case I with $\mathrm{Re}=18, T=1$.]
    {\includegraphics[width=\linewidth,height=0.3\linewidth]{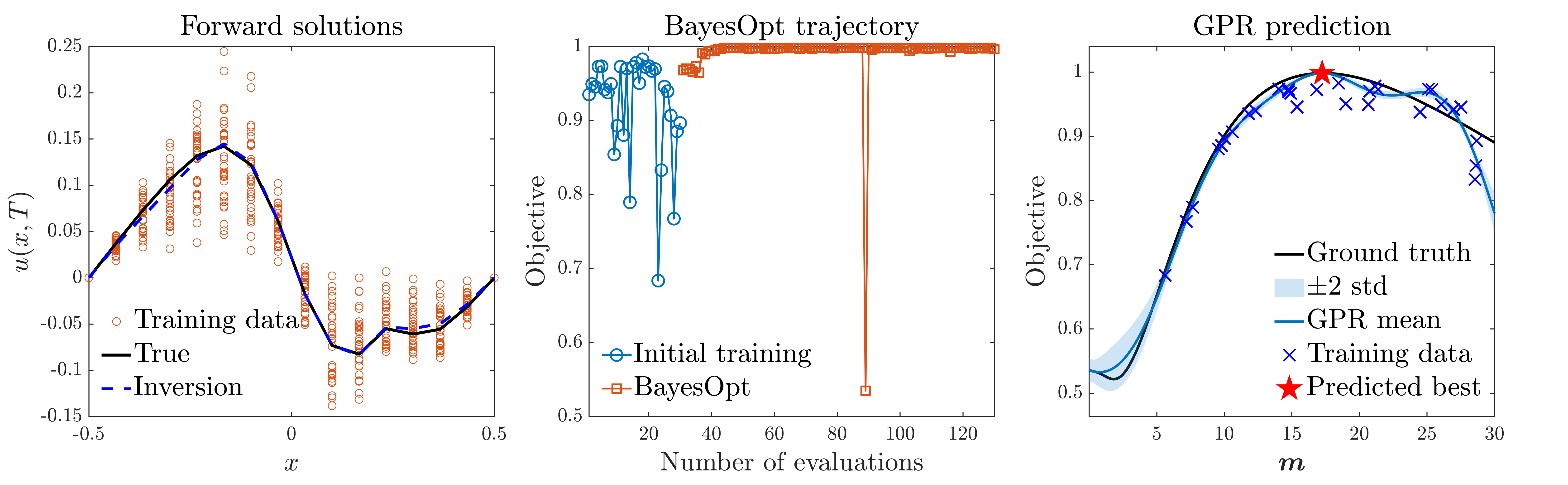}}\\
    \subfloat[Case II with $\mathrm{Re}=14, T=2$.]
    {\includegraphics[width=\linewidth,height=0.3\linewidth]{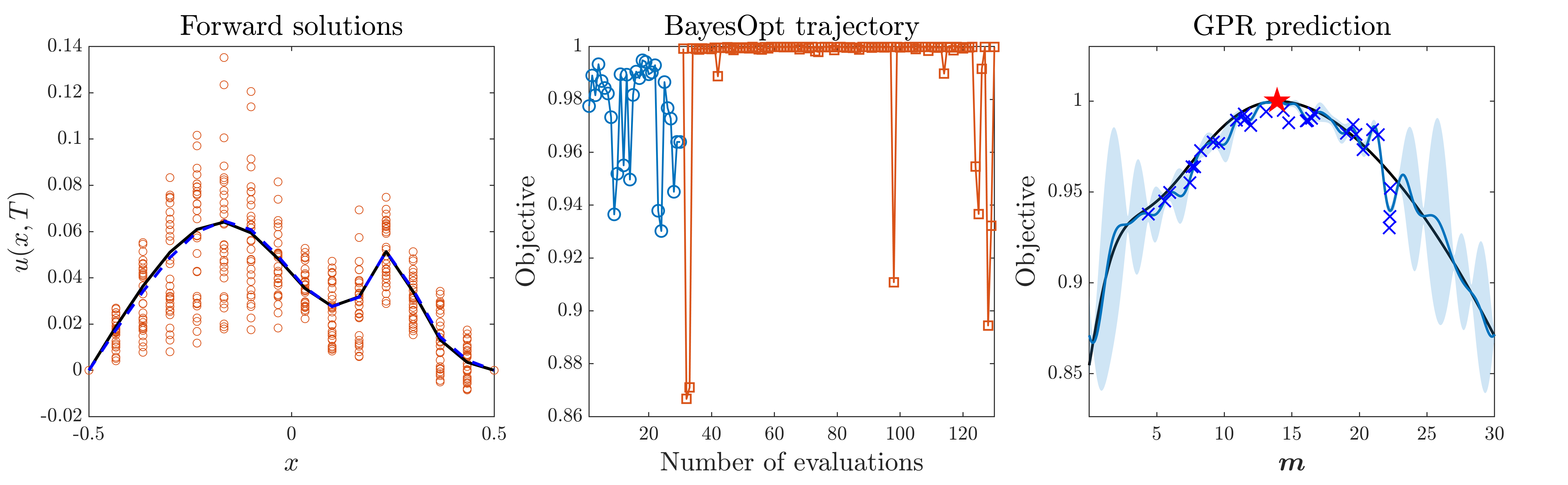}}\\
    \subfloat[Case III with $\mathrm{Re}=10, T=3$.]
    {\includegraphics[width=\linewidth,height=0.3\linewidth]{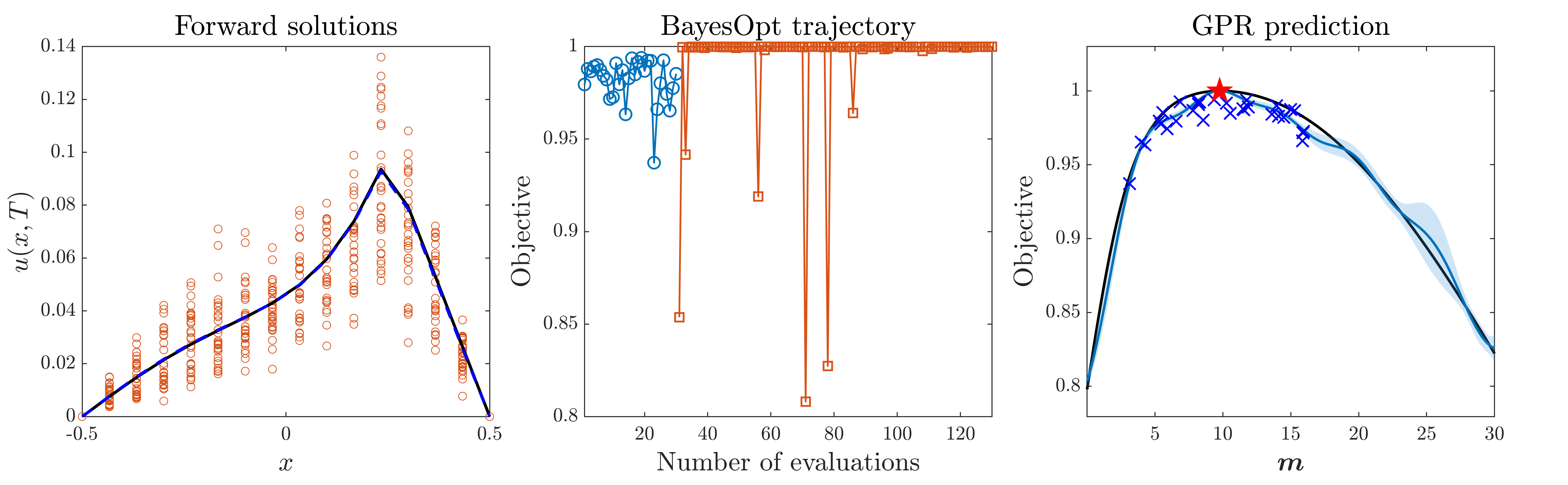}}\\
    \caption{The inversion results for the Burgers equation using the normalized lifted loss $\mathcal{L}^{Y}_{\rm norm}(\bmm)$. }
    \label{fig:Burger_car_loss}
\end{figure}

\begin{figure}[ht]
    \centering
    \subfloat[Case I with $\mathrm{Re}=18, T=1$.]
    {\includegraphics[width=\linewidth,height=0.3\linewidth]{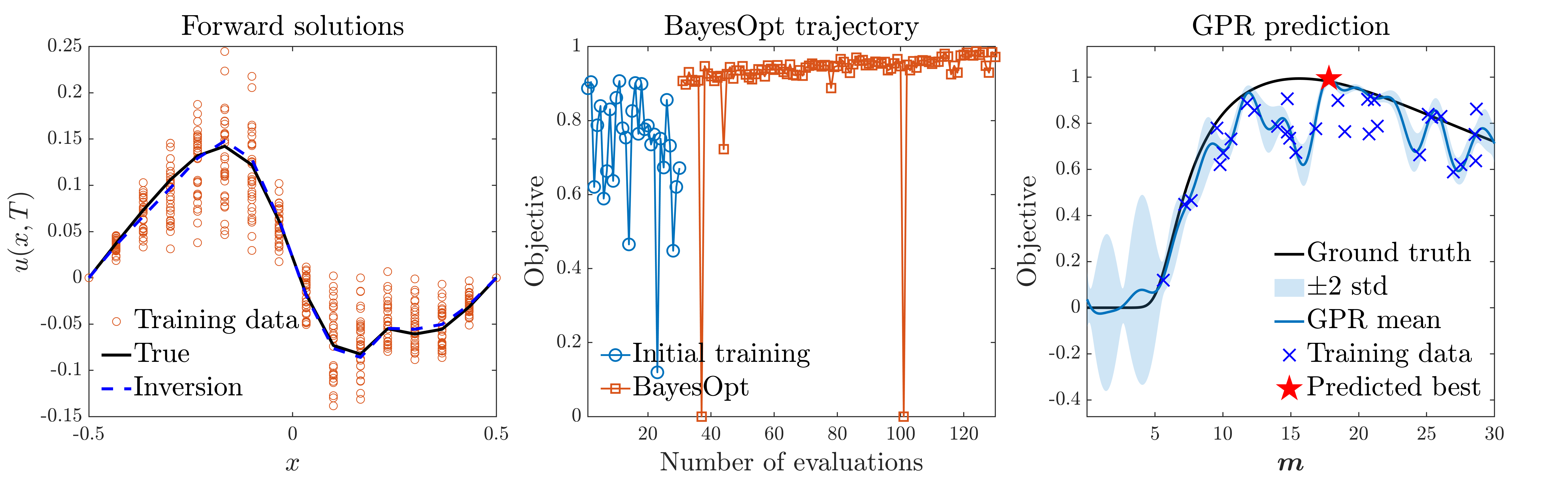}}\\
    \subfloat[Case II with $\mathrm{Re}=14, T=2$.]
    {\includegraphics[width=\linewidth,height=0.3\linewidth]{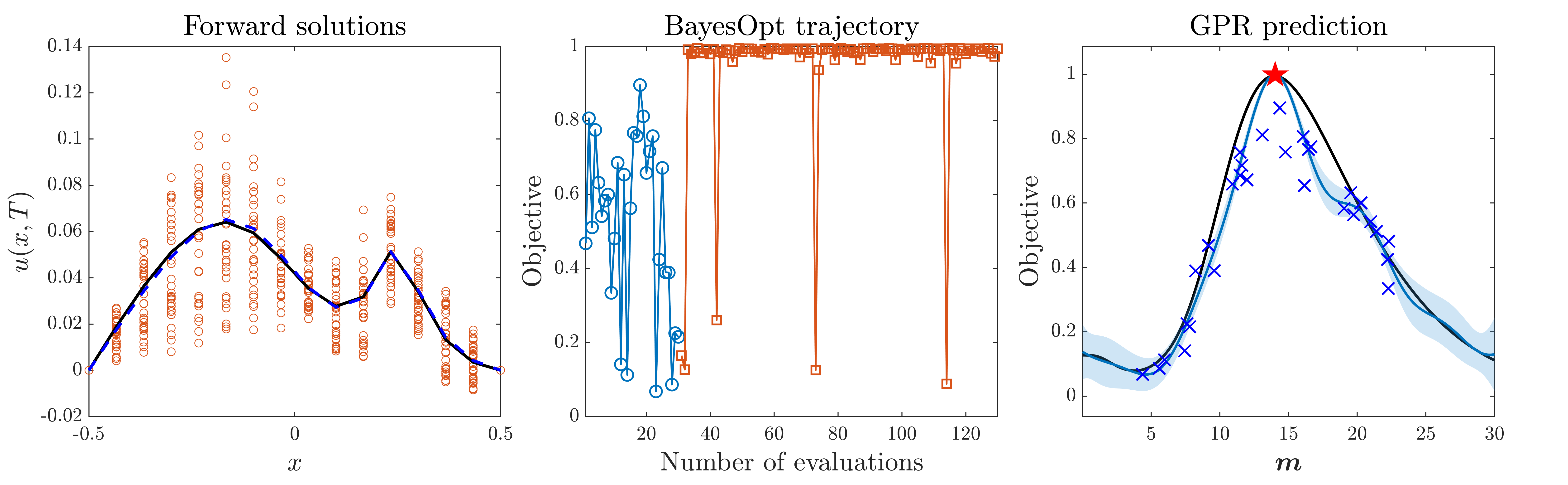}}\\
    \subfloat[Case III with $\mathrm{Re}=10, T=3$.]
    {\includegraphics[width=\linewidth,height=0.3\linewidth]{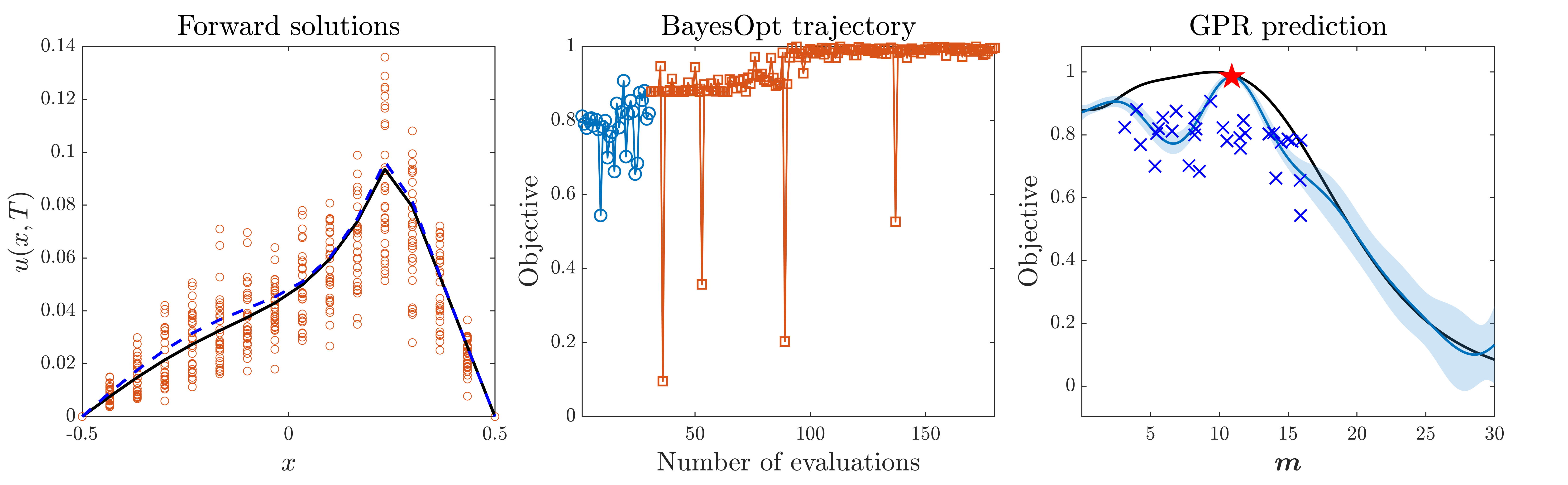}}\\
    \caption{The inversion results for the Burgers equation using the normalized loss $\mathcal{L}^{u}_{\rm norm}(\bmm)$. }
    \label{fig:Burger_norm_loss}
\end{figure}

The estimated parameters $\bmm_{\rm opt}=\arg\max_{\bmm}\mathbb{E}[f(\bmm)\mid\mathcal{D}]$, obtained from the trained Gaussian process regression (GPR) surrogate, and the corresponding relative $L^2$ errors between the inversion solutions and the true solutions are reported in \cref{tab:Burger_results_2}. Moreover, \cref{fig:Burger_phys_loss,fig:Burger_norm_loss,fig:Burger_car_loss} present the exact baseline results for the three loss functions, with panels (a)–(c) corresponding to Cases I–III, respectively. For each case, the left panel compares the ground truth solution corresponding to $\bmm^*$ with the inversion solution obtained using $\bmm_{\mathrm{opt}}$, the middle panel shows the detailed optimization trajectory, and the right panel presents the exact objective function together with the GPR surrogate. In each case, the BayesOpt procedure exhibits convergence, and the prediction closely approximates the exact objective function in a neighborhood of $\bmm^*$. Although the objective values associated with different parameters are more compressed under the normalized lifted loss, its optimization performance remains comparable to that of the physical loss. These results provide numerical support for extending the proposed framework to nonlinear problems through Carleman linearization and quantum-state-based normalized loss functions.

\bibliographystyle{siamplain}

\begin{thebibliography}{10}

\bibitem{al2011computing}
{\sc A.~H. Al-Mohy and N.~J. Higham}, {\em Computing the action of the matrix
  exponential, with an application to exponential integrators}, SIAM J. Sci.
  Comput., 33 (2011), pp.~488--511.

\bibitem{alexanderian2021optimal}
{\sc A.~Alexanderian}, {\em Optimal experimental design for
  infinite-dimensional {Bayesian} inverse problems governed by pdes: A review},
  Inverse Probl., 37 (2021), p.~043001.

\bibitem{alexanderian2014optimal}
{\sc A.~Alexanderian, N.~Petra, G.~Stadler, and O.~Ghattas}, {\em A-optimal
  design of experiments for infinite-dimensional {Bayesian} linear inverse
  problems with regularized $\ell_0$-sparsification}, SIAM J. Sci. Comput., 36
  (2014), pp.~A2122--A2148.

\bibitem{alexis2026infinite}
{\sc M.~Alexis, L.~Lin, G.~Mnatsakanyan, C.~Thiele, and J.~Wang}, {\em Infinite
  quantum signal processing for arbitrary {Szeg{\H{o}}} functions}, Commun.
  Pure Appl. Math., 79 (2026), pp.~123--174.

\bibitem{AlipanahZhangYao2025}
{\sc H.~Alipanah, F.~Zhang, Y.-X. Yao, et~al.}, {\em Quantum dynamics
  simulation of the advection-diffusion equation}, Phys. Rev. Res., 7 (2025),
  p.~043318.

\bibitem{AlkadriKharaziWhaleyMandadapu2025}
{\sc A.~M. Alkadri, T.~D. Kharazi, K.~B. Whaley, and K.~K. Mandadapu}, {\em A
  quantum algorithm for the finite element method}, arXiv preprint
  arXiv:2510.18150,  (2025).

\bibitem{an2023linear}
{\sc D.~An, J.-P. Liu, and L.~Lin}, {\em Linear combination of {Hamiltonian}
  simulation for nonunitary dynamics with optimal state preparation cost},
  Phys. Rev. Lett., 131 (2023), p.~150603.

\bibitem{an2025quantum}
{\sc D.~An, J.-P. Liu, D.~Wang, and Q.~Zhao}, {\em Quantum differential
  equation solvers: Limitations and fast-forwarding}, Commun. Math. Phys., 406
  (2025), p.~189.

\bibitem{arute2019quantum}
{\sc F.~Arute, K.~Arya, R.~Babbush, D.~Bacon, J.~C. Bardin, R.~Barends,
  R.~Biswas, S.~Boixo, F.~G. S.~L. Brandao, D.~A. Buell, et~al.}, {\em Quantum
  supremacy using a programmable superconducting processor}, Nature, 574
  (2019), pp.~505--510.

\bibitem{bergholm2018pennylane}
{\sc V.~Bergholm, J.~Izaac, M.~Schuld, C.~Gogolin, S.~Ahmed, V.~Ajith, M.~S.
  Alam, G.~Alonso-Linaje, B.~AkashNarayanan, A.~Asadi, et~al.}, {\em
  {PennyLane}: Automatic differentiation of hybrid quantum--classical
  computations}, arXiv preprint arXiv:1811.04968,  (2018).

\bibitem{Berry2014HighOrder}
{\sc D.~W. Berry}, {\em High-order quantum algorithm for solving linear
  differential equations}, Journal of Physics A: Mathematical and Theoretical,
  47 (2014), p.~105301.

\bibitem{berry2015simulating}
{\sc D.~W. Berry, A.~M. Childs, R.~Cleve, R.~Kothari, and R.~D. Somma}, {\em
  Simulating {Hamiltonian} dynamics with a truncated {Taylor} series}, Phys.
  Rev. Lett., 114 (2015), p.~090502.

\bibitem{berry2017quantum}
{\sc D.~W. Berry, A.~M. Childs, A.~Ostrander, and G.~Wang}, {\em Quantum
  algorithm for linear differential equations with exponentially improved
  dependence on precision}, Commun. Math. Phys., 356 (2017), pp.~1057--1081.

\bibitem{bosch2025quantum}
{\sc C.~B{\"o}sch, M.~Schade, G.~Aloisi, S.~D. Keating, and A.~Fichtner}, {\em
  Quantum wave simulation with sources and loss functions}, Phys. Rev. Res.‌,
  7 (2025), p.~033225.

\bibitem{chen2026timedependenthamiltoniansimulationoptimal}
{\sc B.~Chen, M.~Gao, X.~Wang, and S.~Zhou}, {\em Time-dependent hamiltonian
  simulation with optimal query complexity}, 2026,
  \url{https://arxiv.org/abs/2608.06094}.

\bibitem{chiappetta2026efficient}
{\sc M.~Chiappetta, M.~Carraturo, A.~Ra{\ss}loff, M.~K{\"a}stner, and
  F.~Auricchio}, {\em An efficient {Bayesian} framework for inverse problems
  via optimization and inversion: Surrogate modeling, parameter inference, and
  uncertainty quantification}, arXiv preprint arXiv:2602.04537,  (2026).

\bibitem{childs2017quantum}
{\sc A.~M. Childs, R.~Kothari, and R.~D. Somma}, {\em Quantum algorithm for
  systems of linear equations with exponentially improved dependence on
  precision}, SIAM J. Comput., 46 (2017), pp.~1920--1950.

\bibitem{childs2021high}
{\sc A.~M. Childs, J.-P. Liu, and A.~Ostrander}, {\em High-precision quantum
  algorithms for partial differential equations}, Quantum, 5 (2021), p.~574.

\bibitem{divincenzo1995quantum}
{\sc D.~P. DiVincenzo}, {\em Quantum computation}, Science, 270 (1995),
  pp.~255--261.

\bibitem{DongLiXue2025}
{\sc D.~Dong, Y.~Li, and J.~Xue}, {\em A quantum algorithm for linear
  autonomous differential equations via {Padé} approximation}, Quantum, 9
  (2025), p.~1770.

\bibitem{dong2024robust}
{\sc Y.~Dong, L.~Lin, H.~Ni, and J.~Wang}, {\em Robust iterative method for
  symmetric quantum signal processing in all parameter regimes}, SIAM J. Sci.
  Comput., 46 (2024), pp.~A2951--A2971.

\bibitem{feynman2018simulating}
{\sc R.~P. Feynman}, {\em Simulating physics with computers}, in Feynman and
  Computation, CRC Press, 2018, pp.~133--153.

\bibitem{frazier2018bayesian}
{\sc P.~I. Frazier}, {\em {Bayesian} optimization}, in Recent Advances in
  Optimization and Modeling of Contemporary Problems, INFORMS, 2018,
  pp.~255--278.

\bibitem{garnett2023bayesian}
{\sc R.~Garnett}, {\em {Bayesian} Optimization}, Cambridge University Press,
  2023.

\bibitem{gilyen2019quantum}
{\sc A.~Gily{\'e}n, Y.~Su, G.~H. Low, and N.~Wiebe}, {\em Quantum singular
  value transformation and beyond: Exponential improvements for quantum matrix
  arithmetics}, in Proceedings of the 51st Annual {ACM} {SIGACT} Symposium on
  Theory of Computing, 2019, pp.~193--204.

\bibitem{gutierrez2026quantum}
{\sc X.~Guti{\'e}rrez, L.~Laneve, and M.~Sanz}, {\em Quantum eigenvalue
  transformations for arbitrary matrices}, arXiv preprint arXiv:2604.19688,
  (2026).

\bibitem{haah2016sample}
{\sc J.~Haah, A.~W. Harrow, Z.~Ji, X.~Wu, and N.~Yu}, {\em Sample-optimal
  tomography of quantum states}, in Proceedings of the Forty-Eighth Annual
  {ACM} Symposium on Theory of Computing, 2016, pp.~913--925.

\bibitem{harrow2009quantum}
{\sc A.~W. Harrow, A.~Hassidim, and S.~Lloyd}, {\em Quantum algorithm for
  linear systems of equations}, Phys. Rev. Lett., 103 (2009), p.~150502.

\bibitem{idier2013bayesian}
{\sc J.~Idier}, {\em {Bayesian} Approach to Inverse Problems}, John Wiley \&
  Sons, 2013.

\bibitem{jennings2025quantumalgorithmsgeneralnonlinear}
{\sc D.~Jennings, K.~Korzekwa, M.~Lostaglio, A.~T. Sornborger, Y.~Subasi, and
  G.~Wang}, {\em Quantum algorithms for general nonlinear dynamics based on the
  carleman embedding}, 2025, \url{https://arxiv.org/abs/2509.07155}.

\bibitem{jennings2026quantumkoopmanalgorithms}
{\sc D.~Jennings, K.~Korzekwa, M.~Lostaglio, and G.~Wang}, {\em Quantum koopman
  algorithms}, 2026, \url{https://arxiv.org/abs/2605.19054}.

\bibitem{JinLiu2024}
{\sc S.~Jin and N.~Liu}, {\em Analog quantum simulation of partial differential
  equations}, Quantum Sci. Technol., 9 (2024), p.~035047.

\bibitem{Jin2024Quantumalgorithmsfornonlinear}
{\sc S.~Jin and N.~Liu}, {\em Quantum algorithms for nonlinear partial
  differential equations}, Bulletin des Sciences Mathématiques, 194 (2024),
  p.~103457.

\bibitem{JinLiu2026}
{\sc S.~jin and N.~Liu}, {\em Quantum algorithms for viscosity solutions to
  nonlinear {H}amilton–{J}acobi equations based on an entropy penalization
  method}, Proc. Natl. Acad. Sci. U.S.A‌, 123 (2026), p.~e2607144123.

\bibitem{jin2026quantumalgorithmsyoungmeasures}
{\sc S.~Jin, N.~Liu, M.~Lukacova-Medvidova, and Y.~Yuan}, {\em Quantum
  algorithms for young measures: applications to nonlinear partial differential
  equations}, 2026, \url{https://arxiv.org/abs/2604.11825}.

\bibitem{jin2025inhomogeneous}
{\sc S.~Jin, N.~Liu, and C.~Ma}, {\em On {Schr{\"o}dingerization}-based quantum
  algorithms for linear dynamical systems with inhomogeneous terms}, SIAM J.
  Numer. Anal., 63 (2025), pp.~1861--1885.

\bibitem{jin2025schrodingerization}
{\sc S.~Jin, N.~Liu, and C.~Ma}, {\em {Schr{\"o}dingerization}-based
  computationally stable algorithms for ill-posed problems in partial
  differential equations}, SIAM J. Sci. Comput., 47 (2025), pp.~B976--B1000.

\bibitem{jin2023quantum}
{\sc S.~Jin, N.~Liu, and Y.~Yu}, {\em Quantum simulation of partial
  differential equations: Applications and detailed analysis}, Phys. Rev. A,
  108 (2023), p.~032603.

\bibitem{jin2024quantum}
{\sc S.~Jin, N.~Liu, and Y.~Yu}, {\em Quantum simulation of partial
  differential equations via {Schr{\"o}dingerization}}, Phys. Rev. Lett., 133
  (2024), pp.~230602--230602.

\bibitem{Joseph2020KoopmanvonNeumann}
{\sc I.~Joseph}, {\em Koopman–von neumann approach to quantum simulation of
  nonlinear classical dynamics}, Physical Review Research, 2 (2020).

\bibitem{li2025nearly}
{\sc L.~Li and J.~Luo}, {\em Nearly optimal circuit size for sparse quantum
  state preparation}, in 52nd International Colloquium on Automata, Languages,
  and Programming (ICALP 2025), vol.~334 of Leibniz International Proceedings
  in Informatics (LIPIcs), Schloss Dagstuhl -- Leibniz-Zentrum f{\"u}r
  Informatik, 2025, pp.~113:1--113:19.

\bibitem{Li2026}
{\sc X.~Li}, {\em A quantum path to partial differential equations}, arXiv
  preprint arXiv:2607.09639,  (2026).

\bibitem{liu2021efficient}
{\sc J.-P. Liu, H.~{\O}. Kolden, H.~K. Krovi, N.~F. Loureiro, K.~Trivisa, and
  A.~M. Childs}, {\em Efficient quantum algorithm for dissipative nonlinear
  differential equations}, Proc. Natl. Acad. Sci. U.S.A., 118 (2021),
  p.~e2026805118.

\bibitem{low2017optimal}
{\sc G.~H. Low and I.~L. Chuang}, {\em Optimal {Hamiltonian} simulation by
  quantum signal processing}, Phys. Rev. Lett., 118 (2017), p.~010501.

\bibitem{Low2019Hamiltonian}
{\sc G.~H. Low and I.~L. Chuang}, {\em Hamiltonian simulation by qubitization},
  Quantum, 3 (2019), p.~163.

\bibitem{low2025optimalquantumsimulationlinear}
{\sc G.~H. Low and R.~D. Somma}, {\em Optimal quantum simulation of linear
  non-unitary dynamics}, 2025, \url{https://arxiv.org/abs/2508.19238}.

\bibitem{low2026quantum}
{\sc G.~H. Low and Y.~Su}, {\em Quantum eigenvalue processing}, SIAM J.
  Comput., 55 (2026), pp.~135--215.

\bibitem{ma2026schrodingerization}
{\sc C.~Ma, S.~Jin, N.~Liu, K.~Wang, and L.~Zhang}, {\em
  {Schr{\"o}dingerization}-based quantum circuits for {Maxwell's} equations
  with time-dependent source terms}, J. Sci. Comput., 108 (2026), p.~43.

\bibitem{mao2024toward}
{\sc R.~Mao, G.~Tian, and X.~Sun}, {\em Toward optimal circuit size for sparse
  quantum state preparation}, Phys. Rev. A, 110 (2024), p.~032439.

\bibitem{MengZhongXu2024}
{\sc Z.~Meng, J.~Zhong, S.~Xu, et~al.}, {\em Simulating unsteady flows on a
  superconducting quantum processor}, Commun. Phys., 7 (2024), p.~349.

\bibitem{MontanaroPallister2016}
{\sc A.~Montanaro and S.~Pallister}, {\em Quantum algorithms and the finite
  element method}, Phys. Rev. A, 93 (2016), p.~032324.

\bibitem{motlagh2024generalized}
{\sc D.~Motlagh and N.~Wiebe}, {\em Generalized quantum signal processing}, PRX
  Quantum, 5 (2024), p.~020368.

\bibitem{nielsen2010quantum}
{\sc M.~A. Nielsen and I.~L. Chuang}, {\em Quantum Computation and Quantum
  Information}, Cambridge University Press, 2010.

\bibitem{ramacciotti2024simple}
{\sc D.~Ramacciotti, A.~I. Lefterovici, and A.~F. Rotundo}, {\em Simple quantum
  algorithm to efficiently prepare sparse states}, Phys. Rev. A, 110 (2024),
  p.~032609.

\bibitem{rasmussen2003gaussian}
{\sc C.~E. Rasmussen}, {\em {Gaussian} processes in machine learning}, in
  Summer School on Machine Learning, Springer, 2003, pp.~63--71.

\bibitem{SchadeBöschHaplaFichtner2024}
{\sc M.~Schade, C.~Bösch, V.~Hapla, and A.~Fichtner}, {\em A quantum computing
  concept for 1-d elastic wave simulation with exponential speedup}, Geophys.
  J. Int., 238 (2024), pp.~321--333.

\bibitem{MoralesPiraSchleich2025}
{\sc M.~Schade, C.~Bösch, V.~Hapla, and A.~Fichtner}, {\em Quantum linear
  system solvers: A survey of algorithms and applications}, Rev. Mod. Phys., 98
  (2026), p.~025005.

\bibitem{schuld2021quantum}
{\sc M.~Schuld and F.~Petruccione}, {\em Quantum models as kernel methods}, in
  Machine Learning with Quantum Computers, Springer, 2021, pp.~217--245.

\bibitem{stuart2010inverse}
{\sc A.~M. Stuart}, {\em Inverse problems: A {Bayesian} perspective}, Acta
  Numer., 19 (2010), pp.~451--559.

\bibitem{tarantola2005inverse}
{\sc A.~Tarantola}, {\em Inverse Problem Theory and Methods for Model Parameter
  Estimation}, SIAM, 2005.

\bibitem{wu2025quantum}
{\sc H.-C. Wu, J.~Wang, and X.~Li}, {\em Quantum algorithms for nonlinear
  dynamics: Revisiting {Carleman} linearization with no dissipative
  conditions}, SIAM J. Sci. Comput., 47 (2025), pp.~A943--A970.

\bibitem{ying2022stable}
{\sc L.~Ying}, {\em Stable factorization for phase factors of quantum signal
  processing}, Quantum, 6 (2022), p.~842.

\bibitem{zhang2022quantum}
{\sc X.-M. Zhang, T.~Li, and X.~Yuan}, {\em Quantum state preparation with
  optimal circuit depth: Implementations and applications}, Phys. Rev. Lett.,
  129 (2022), p.~230504.

\bibitem{zhao2019quantum}
{\sc Z.~Zhao, J.~K. Fitzsimons, and J.~F. Fitzsimons}, {\em Quantum-assisted
  {Gaussian} process regression}, Phys. Rev. A, 99 (2019), p.~052331.

\end{thebibliography}

\end{document}